\documentclass[10pt,reqno]{amsart}
\usepackage{amsmath,amssymb,amsfonts,amsthm}
\usepackage[left=3cm, right=3cm, top=2.8cm, bottom=2.8cm]{geometry}
\usepackage[utf8]{inputenc}
\usepackage[T1]{fontenc}
\usepackage{xcolor}
\usepackage{graphicx}

\newtheorem{theorem}{Theorem}
\newtheorem{lemma}[theorem]{Lemma}
\newtheorem{proposition}[theorem]{Proposition}
\newtheorem{corollary}[theorem]{Corollary}
\newtheorem{definition}[theorem]{Definition}
\newtheorem{remark}[theorem]{Remark}

\theoremstyle{plain}

\begin{document}

\title[A New approach for the unsteady Stokes equations]{A New approach for the unsteady Stokes equations with time fractional derivative in Bounded Domains}


\author[J. C. Oyola Ballesteros]{Juan C. Oyola Ballesteros}
\address[Juan C. Oyola Ballesteros]{Department of Mathematics, Federal University of Santa Catarina, Florian\'{o}polis - SC, Brazil\vspace*{0.3cm}}
\email[]{jcoyolaba@unal.edu.co}
\author[P. M. Carvalho-Neto]{Paulo M. Carvalho-Neto}
\address[Paulo M. de Carvalho Neto]{Department of Mathematics, Federal University of Santa Catarina, Florian\'{o}polis - SC, Brazil\vspace*{0.3cm}}
\email[]{paulo.carvalho@ufsc.br}






\begin{abstract}
In this work, we introduce a novel variational framework for the study of the unsteady Stokes equations in a bounded, open Lipschitz domain $\Omega \subset \mathbb{R}^{n}$, involving a Caputo fractional derivative in time. The nonlocal nature of the fractional derivative presents significant analytical challenges, making classical methods, such as the Faedo–Galerkin approach, inadequate in their standard form for a full analysis of the problem. To address these difficulties, we develop and rigorously analyze new functional spaces specifically designed for the fractional setting. These spaces allow us to reformulate weak solutions in a manner consistent with the fractional dynamics, thereby enabling the successful implementation of a generalized Galerkin scheme. Our formulation not only extends the classical theory but also provides a foundation for the study of broader classes of fractional partial differential equations in fluid mechanics and related areas.
\end{abstract}

\maketitle

\section{Introduction}

One of the most renowned mathematical models in fluid dynamics is the Navier-Stokes equations. These equations aim to determine the velocity and pressure fields within a fluid. They are among the most widely used equations, describing the physics of numerous phenomena of both economic and academic interest. The Navier-Stokes equations find applications in modeling weather patterns, ocean currents, water flow in canoes, and many other areas. These equations can be derived directly from Newton's laws under the assumption of incompressibility, where pressure does not affect the fluid’s volume. For more details, we refer to \cite[Chapter 1]{CRdJDgApp}.

A significant breakthrough in the theory of partial differential equations was the introduction of the concept of {weak solutions} by J. Leray, as presented in \cite{JLEs}, particularly in the context of the Navier--Stokes equations. Leray’s theory establishes the existence of solutions in a new variational framework, which may potentially be irregular. This approach relies on energy estimates and specific limiting procedures involving the weak and weak-$\ast$ topologies in certain Bochner--Lebesgue spaces (see \cite{CfSt, CfGPCom, CfRtGv} for classical references).

On the other hand, fractional calculus, with its applications to solving fractional-order differential equations (see \cite{KiSrTrTh,FmFrac,SaKiMaFrac} for early developments), has become essential for addressing problems in fields such as fluid dynamics, porous media, control theory, and dynamical systems (see \cite{metzler, HsRmAbTFnGen, WRsWwFr, WwTF} for classical applications). Unlike standard calculus, which deals with derivatives and integrals of integer order, fractional calculus extends these operations to real or complex orders. 

Consequently, the combined importance of the two topics mentioned above highlights the growing interest in recent studies on the Navier--Stokes equations with Caputo fractional derivatives of order $\alpha \in (0,1)$ in time. These works aim to generalize the classical problem and explore the limiting case $\alpha = 1$ (see \cite{CarPl1, PaRFL,LiLiu1, ShSa1, ZhPe1, ZhPe2, ZoLvWu1} for some examples).

However, it is important to note that most existing studies focus on establishing the existence and uniqueness of mild solutions for the abstract Cauchy problem associated with the fractional Navier--Stokes equations. In contrast, only a few works address the existence of weak solutions, mainly due to the technical difficulties of applying the Faedo--Galerkin method in the fractional context. For instance, \cite{ZhPe2} defines and proves the existence of a weak solution to the Navier--Stokes equations with a Caputo fractional derivative of order $\alpha \in (0,1)$, using the same definition as in the classical case when $\alpha = 1$. Similarly, \cite{LiLiu1} introduces a new functional space (see \cite[Item (4.5)]{LiLiu1}) to study weak solutions of the multidimensional fractional Burgers equations (a special case of the compressible fractional Navier--Stokes equations) while still relying on the classical definition (see \cite[Definition 5.1]{LiLiu1}).

In conclusion, this work aims to go beyond traditional approaches by proposing a new formulation for weak solutions of partial differential equations with a Caputo fractional derivative in time. We not only introduce this new formulation, but also analyze a modified version of the functional spaces introduced in \cite{LiLiu1}, proving why they are more appropriate for the variational framework we adopt.

Although all discussions up to this point have consistently led back to the fractional Navier--Stokes problem, it is important to clarify that the present work does not aim to establish the existence or uniqueness of its solutions. Rather, our focus lies in introducing new functional spaces and proving key properties related to them, inspired by those proposed in \cite{LiLiu1}. Due to the additional difficulties introduced by the nonlinear term in the Navier--Stokes equations, we restrict our attention in this initial study to its classical linearization: the unsteady Stokes equations. This choice simplifies the analysis while still serving to illustrate the proposed framework. The treatment of the full nonlinear case is left for future work.

\subsection{A Novel Weak Formulation}

Let us first recall the fractional order approach to the unsteady Stokes equations. Let $\Omega$ be a bounded Lipschitz subset of $\mathbb{R}^n$, for some $n\geq2$, and let $\alpha\in(0,1)$, $T>0$, and $\nu>0$ be fixed constants. Consider $f:[0,T]\times\Omega\rightarrow\mathbb{R}^n$ as the forcing term and $u_0:\Omega\rightarrow\mathbb{R}^n$ as the initial condition. The unsteady Stokes equations with a Caputo fractional derivative in the time variable are given by
 \begin{equation}\label{navierstokes01}\left\{ \begin{array}{rll}
    	cD_t^\alpha u(t,x)-\nu\Delta u(t,x) +\nabla p(t,x)=f(t,x),&&\textrm{ in } (0,T)\times\Omega,\\
    	\operatorname{div} u(t,x)=0,&&\textrm{ in } (0,T)\times\Omega,\\
    	u(t,x)=0,&&\textrm{ on } [0,T]\times\partial\Omega,\\
	    u(0,x)=u_{0}(x),&&\textrm{ in } \Omega,
	    \end{array}\right.
   \end{equation}
where $cD_t^\alpha$ stands for the Caputo fractional derivative of order $\alpha$ (for more details see Section \ref{secfrac}).

It is common to use some classical notations in the study of the unsteady Stokes equation (see \cite{RTemamb} for more details). Let us first recall the function space 
$$\mathcal{V}:=\left\{v\in\big[\mathcal{D}(\Omega)\big]^n: \operatorname{div} \; v=0\right\},$$
where $\mathcal{D}(\Omega)$ denotes the test functions in $\Omega$. Also consider
\begin{equation}\label{202503141707}
	H:=\overline{\mathcal{V}}^{\mathbb{L}^{2}(\Omega)}
	=\left\{v\in\mathbb{L}^{2}(\Omega): \operatorname{div} v=0\textrm{ in }\Omega\textrm{ and }v\cdot \vec{n}=0\textrm{ on }\partial\Omega \right\},
\end{equation}
which with the induced topology of $\mathbb{L}^{2}(\Omega):=\big[{L}^{2}(\Omega)\big]^n$ becomes a Hilbert space, and finally
\begin{equation}\label{202503141708}
	V:=\overline{\mathcal{V}}^{\mathbb{H}_{0}^{1}(\Omega)}=\{v\in\mathbb{H}_{0}^{1}(\Omega):\; \operatorname{div} v=0\textrm{ and }\Omega\},
\end{equation}
that equipped with the inner product inherited from $\mathbb{H}_{0}^{1}(\Omega)=\big[{H}_{0}^{1}(\Omega)\big]^n$ is also a Hilbert space. 

Now suppose that $u$ and $p$ are sufficiently regular solutions of \eqref{navierstokes01}, and that $f$ is at least in $L^{2}(0,T;V^\prime)$. Then, if $\eta\in \mathcal{V}$ we have
\begin{equation}\label{SF2}
(cD_t^\alpha u(t),\eta)_H+\nu(u(t),\eta)_V=\langle f(t),\eta\rangle_{V^\prime,V},	\textrm{ for a.e. }t\in[0,T].
\end{equation}
Above the symbol $\langle\cdot,\cdot\rangle_{V^\prime,V}$ denotes the duality pairing between $V^\prime$ and $V$.

First, let $\eta = u(t)$ in \eqref{SF2}. By performing a formal calculation and {applying \cite[Corollary 19]{CarFrOyTo1}}, we derive the energy inequality
\begin{equation}\label{energysf01}
\frac{1}{2} c D_t^\alpha \|u(t)\|^2_H + \nu \|u(t)\|^2_V \leq \langle f(t), u(t) \rangle_{V^\prime,V}, \text{ for a.e. } t \in [0,T],
\end{equation}
which becomes an equality when $\alpha = 1$.

Following the classical ideas introduced by Leray in \cite{JLEs, JLEs2, JLEs3}, below we present a series of formal arguments to explore the regularity that can be expected from the weak solution of the unsteady Stokes equations, which is derived from \eqref{energysf01}.

\begin{itemize}
\item[(i)] Recall that when $\alpha = 1$, as is classically justified in the literature, integrating both sides of \eqref{energysf01} from $0$ to $t$ and applying Young's inequality allows us to deduce that
\begin{equation*}
\|u(t)\|^2_{H}+\nu\int_0^t\|u(s)\|^2_V\,ds
\leq \|u_0\|^2_{H}+\frac{1}{\nu} \int_0^t\|f(s)\|^2_{V^\prime}\,ds,	\textrm{ for a.e. } t\in[0,T].
\end{equation*}
The above inequality suggests that requiring $u \in L^{2}(0,T;V) \cap L^{\infty}(0,T;H)$ is a reasonable regularity condition for the classical weak solution. This regularity serves as a foundation for defining weak solutions to the unsteady Stokes equations.\vspace*{0.2cm}

\item[(ii)] However, when $\alpha \in (0,1)$, these function spaces may no longer be ideal. Mimicking the classical formulation, we integrate both sides of \eqref{energysf01} from \(0\) to \(t\), use Proposition \ref{Pro1}, and apply Young's inequality to obtain
\begin{multline*}
\qquad\quad J^{1-\alpha}_t\|u(t)\|^2_{H}+\nu \int_0^t\|u(s)\|^2_V\,ds
\leq \frac{t^{1-\alpha}}{\Gamma(2-\alpha)}\|u_0\|^2_{H} + \frac{1}{\nu} \int_0^t\|f(s)\|^2_{V^\prime}\,ds,	\quad \textrm{for a.e. } t\in[0,T].
\end{multline*}
From this, we derive the estimate
\begin{equation}\label{202505211721}
J^{1-\alpha}_t\|u(t)\|^2_{H}\Big|_{t=T}+\nu \|u\|^2_{L^2(0,T;V)}
\leq \frac{T^{1-\alpha}}{\Gamma(2-\alpha)}\|u_0\|^2_{H} + \frac{1}{\nu} \|f\|^2_{L^2(0,T;V^\prime)}.
\end{equation}
It follows from \eqref{202505211721} that the solution $u$ should belong to $L^{2}(0,T;V)$ and that $J^{1-\alpha}_t\|u(t)\|^2_{H}$, evaluated at $t=T$, must be finite. \end{itemize}

The discussion above indicates the need to consider a new subspace of $L^p(0,T;X)$, denoted as $L^p_\alpha(0,T;X)$ for any Banach space $X$ and $p \geq 1$. This subspace consists of Bochner-measurable functions $f:[0,T]\rightarrow X$ for which 
$$\int_0^T(T-s)^{\alpha-1}\|f(s)\|^p_{X}\,ds<\infty.$$

The study of this space, along with establishing the existence and uniqueness of solutions to the weak formulation of the fractional unsteady Stokes equations, becomes our primary focus from this point onward.

\subsection{Outline of the Manuscript}
Section \ref{secfrac} reviews key technical concepts, including Bochner-Lebesgue spaces, Bochner-Sobolev spaces, the Riemann-Liouville fractional integral, and the Caputo fractional derivative. It concludes with the proof of a new version of Carathéodory's existence and uniqueness theorem, specifically adapted to the requirements of our main problem. Section \ref{fracspace} examines the $L^p_\alpha(0,T;X)$ spaces, their connection to the classical $L^p(0,T;X)$ spaces, and the conditions on $X$ necessary to ensure that they are reflexive Banach spaces. Finally, Section \ref{galerfrac} presents a new formulation for the weak solution of the unsteady Stokes equations, proving its existence and uniqueness through the establishment of approximate solutions and the application of a limiting argument.

\section{Theoretical Prerequisites and some Novel Results}\label{secfrac}

We divide this section into two parts. First, we review the classical Bochner-Lebesgue and Bochner-Sobolev spaces and provide a brief overview of fractional-order integration and differentiation. Second, we present new results on the existence and uniqueness of solutions to the fractional Cauchy problem with a Carathéodory map as the forcing term. Notably, this second part is, to the best of the authors' knowledge, a novel contribution to the literature, and we offer a detailed explanation to ensure clarity and understanding.
\subsection{Fractional Calculus on Vector-Valued Functions} From now on, let us assume that $T>0$ is fixed and $X$ is a Banach space. The topics addressed here are mainly connected with fractional calculus, a theory that is well-established in the literature. For instance, we refer to \cite{ArVecValu, Baz, Car, KiSrTrTh, SaKiMaFrac} and references therein for more details. 

For $p\in[1,\infty]$, we use the symbol $L^{p}(0,T;X)$ to represent the set of all Bochner measurable functions $v:[0,T]\to X$ for which $\|v(\cdot)\|_{X}\in L^{p}(0,T;\mathbb{R})$. Moreover, $L^{p}(0,T;X)$ is a Banach space when considered with the norm: 
\begin{itemize}
	\item[(i)] $\|v\|_{L^{p}(0,T;X)}:=\left(\displaystyle\displaystyle\int_0^T\|v(t)\|_{X}^{p}\; dt\right)^{1/p},\textrm{ if }p\in[1,\infty)$;\vspace*{0.3cm}
	\item[(ii)] $\|v\|_{L^{\infty}(0,T;X)}:=\displaystyle \operatorname{ess} \sup\big\{\|v(t)\|_{X}:\textrm{ for almost every }t\in[0,T]\big\}$.
\end{itemize}

Also, for $p\in[1,\infty]$ and $m \in \mathbb{N}$, we use the symbol $W^{m,p}(0,T;X)$ to denote the vector space of all functions $v:[0,T] \to X$ such that, for each $k \in \mathbb{N}$ with $k \leq m$, the derivative $v^{(k)}$ belongs to $L^{p}(0,T; X)$, where $v^{(k)}$ denotes the $k$-th derivative in the weak sense. We call these spaces the Bochner-Sobolev spaces. These are Banach spaces when imbued with the norm:
\begin{itemize}
	\item[(i)] $\|v\|_{W^{m,p}(0,T;X)}:=\displaystyle\sum\limits_{k\leq m}\displaystyle\left(\int_0^T\|v^{(k)}(x)\|_{X}^{p}\; dx\right)^{1/p},\textrm{ if }p\in[1,\infty)$;\vspace*{0.3cm}
	\item[(ii)] $\|v\|_{W^{m,\infty}(0,T;X)}:=\displaystyle \sum\limits_{k\leq m}\operatorname{ess} \sup\big\{\|v^{(k)}(t)\|_{X}:\textrm{ for almost every }t\in[0,T]\big\}$.
\end{itemize}

Let us now introduce the concepts called Riemann-Liouville fractional differentiation and integration of order $\alpha$.
\begin{definition}
	Let $\alpha\in(0,\infty)$ and $v\in L^{1}(0,T;X)$. The Riemann-Liouville fractional integral of order $\alpha$, is denoted by $J_{t}^{\alpha}v(t)$, and is given by
	\begin{equation*}
		J_{t}^{\alpha}v(t):=\dfrac{1}{\Gamma(\alpha)}\int_0^t{(t-s)^{\alpha-1}v(s)}\,ds,\;\;\; \textrm{ for a.e. }t\in[0,T].
	\end{equation*}
\end{definition}

\begin{remark} We find it appropriate to emphasize that the above definition is supported by \cite[Theorem 2.5]{CarFe0}, which guarantees that for any $\alpha \in (0, \infty)$ and $v \in L^{1}(0, T; X)$, the fractional integral $J_t^{\alpha} v(t)$ exists for almost every $t \in [0, T]$ and is Bochner integrable on $[0, T]$.
\end{remark}

\begin{definition}
	Let $\alpha\in(0,1)$ and $v\in L^{1}(0,T;X)$ with $J_{t}^{1-\alpha}v(\cdot)\in W^{1,1}(0,T;X)$. We define the Riemann-Liouville  fractional derivative of order $\alpha$, which is denoted as $D_{t}^{\alpha}v(t)$, by
	\begin{equation*}\begin{array}{lll}
			D_{t}^{\alpha}v(t)
			:=\dfrac{d }{d t}\Big[J_{t}^{1-\alpha}v(t)\Big],
			\;\;\ \textrm{ for a.e. }t\in[0,T].
		\end{array}
	\end{equation*}
\end{definition}

It is important to recall that for $\alpha \in (0,1)$, $\beta \in (-1,\infty)$ and $v(t) = ct^{\beta}$, we have
\begin{equation*}
  D_{t}^{\alpha}v(t) = c \dfrac{\Gamma(\beta + 1)}{\Gamma(1 - \alpha + \beta)} t^{\beta - \alpha},
\end{equation*}
for almost every $t \in [0,1]$. In this example, if $\beta < \alpha$, the Riemann-Liouville fractional derivative of order $\alpha$ of $f$ exhibits a singular behavior at zero. To address this challenge, as well as for various physical and practical reasons, Caputo reformulated the Riemann-Liouville fractional derivative in his well-known paper \cite{CapLin}.

\begin{definition}
	Let $\alpha\in(0,1)$ and $v\in C([0,T],X)$ with $J_{t}^{1-\alpha}v(\cdot)\in W^{1,1}(0,T;X)$. The Caputo fractional derivative of order $\alpha$, which is denoted as $cD_{t}^{\alpha}v(t)$, is given by
	\begin{equation*}
		cD_{t}^{\alpha}v(t):=D_{t}^{\alpha}[v(t)-v(0)], \textrm{ for a.e. }t\in[0,T],
	\end{equation*}
where above $D_{t}^{\alpha}$ represents Riemann-Liouville fractional derivative of order $\alpha$.
\end{definition}

\begin{remark} We point out that the function space $C([0,T];X)$, as considered above, denotes the space of continuous functions $v: [0,T] \to X$ endowed with the norm
$$\|v\|_{C([0,T];X)}:=\sup\big\{\|v(t)\|_{X}:\textrm{ for every }t\in[0,T]\big\}.$$
\end{remark}

Let us now discuss some of the fundamental properties of these fractional derivatives that are critical for several proofs presented in this manuscript.

\begin{proposition}\label{Pro1}
	Let  $\alpha_{1},\alpha_{2}\in(0,1)$ be fixed real numbers and consider functions $v_1\in L^{1}(0,T;X)$ and $v_2\in C([0,T],X) $. Then the following statements are true.
	\vspace*{0.2cm}
	\begin{itemize}
		\item[(i)] $J_{t}^{\alpha_{1}}\big[J_{t}^{\alpha_{2}}v_1(t)\big]=J_{t}^{\alpha_{1}+\alpha_{2}}v_1(t)$,  for a.e. $t\in[0,T]$.\vspace*{0.2cm}
		\item[(ii)] $D_{t}^{\alpha_{1}}\big[J_{t}^{\alpha_{1}}v_1(t)\big]=v_1(t)$,  for a.e. $t\in[0,T]$.\vspace*{0.2cm}
		\item[(iii)] If $J_t^{1-\alpha_1}v_1(\cdot)\in W^{1,1}(0,T;X)$, then
		\begin{equation*}
			J_{t}^{\alpha_{1}}\big[D_{t}^{\alpha_{1}}v_1(t)]=v_1(t)-\dfrac{1}{\Gamma(\alpha_{1})}t^{\alpha_{1}-1}\{J_{s}^{1-\alpha_{1}}v_1(s)\}\big|_{s=0},  \textrm{ for a.e. }t\in[0,T].
		\end{equation*}
		Moreover, if there exists an integral function $v_3\in L^1(0,T;X)$ such that $v_1=J_{t}^{\alpha_{1}}v_3(t)$, then
		\begin{equation*}
			J_{t}^{\alpha_{1}}\big[D_{t}^{\alpha_{1}}v_1(t)\big]=v_1(t),  \textrm{ for a.e. }t\in[0,T].	
		\end{equation*}
		\item[(iv)] $cD_{t}^{\alpha_{1}}\big[J_{t}^{\alpha_{1}}v_2(t)\big]=v_2(t)$,  for a.e. $t\in[0,T]$.\vspace*{0.2cm}
		\item[(v)]
		If $J_t^{1-\alpha_1}v_2(\cdot)\in W^{1,1}(0,b;X)$, then 
\begin{equation*}J_{t}^{\alpha_{1}}\big[cD_{t}^{\alpha_{1}}v_2(t)\big]=v_2(t)-v_2(0),\textrm{ for a.e. }t\in[0,T].\end{equation*}
	\end{itemize}
\end{proposition}
\begin{proof}
 We refere to \cite{Car, KiSrTrTh}.
\end{proof}
\subsection{Fractional Version Of Carath\'{e}odory's Problem}\label{caratsection}
Our aim in this subsection is to establish the existence and uniqueness of a global solution to the Cauchy problem
\begin{equation}
	\left\{\begin{array}{lll}
		\label{EDOF}
		cD_{t}^{\alpha}u(t)&=&G(t,u(t)),\textrm{ for a.e. }t\in[0,T],\\
		u(0)&=&
		u_{0}\in\mathbb{R}^{m},
	\end{array}\right.
\end{equation}
where $\alpha \in (0,1)$, $cD_{t}^{\alpha}$ denotes the Caputo fractional derivative, $T>0$ is a fixed value and $G:[0,T]\times\mathbb{R}^m\to\mathbb{R}^m$ is a Carathéodory map.  It is noteworthy that Lan and Webb in their work \cite[Theorems 5 and 6]{LaWe1} discussed a similar scenario when $m=1$, and the Caratheodory map grows at most linearly on the second variable. However, our focus here lies on addressing a slightly different problem using a different method. To this end, below we introduce some conventional concepts and a classical theorem. For the proofs and further details, we refer to \cite[Theorem 1.27]{RoNon}.
\begin{definition}
	\begin{itemize}
		\item[(i)] We say that $a:[0,T]\times\mathbb{R}^m\rightarrow\mathbb{R}^m$ is a Carath\'{e}odory map if the function $a(\cdot,x):[0,T]\rightarrow\mathbb{R}^m$ is measurable for each $x\in\mathbb{R}^m$, and the function $a(t,\cdot):\mathbb{R}^m\rightarrow\mathbb{R}^m$ is continuous for almost every $t\in[0,T]$.\vspace*{0.2cm}
		\item[(ii)] The Nemytskii Mapping $\mathcal{N}_a$, is a function that maps $v:[0,T]\rightarrow\mathbb{R}^m$ into $\mathcal{N}_a(v):[0,T]\rightarrow\mathbb{R}^m$ which is defined by
		$$\big[\mathcal{N}_a(v)\big](t)=a(t,v(t)).$$
	\end{itemize}
\end{definition}

From the above definition, we may prove the following result.

\begin{proposition}\label{carat} If $a:[0,T]\times\mathbb{R}^m\rightarrow\mathbb{R}^m$ is a Carath\'{e}odory map and $v:[0,T]\rightarrow\mathbb{R}^m$ is measurable, then $t\mapsto\big[\mathcal{N}_a(v)\big](t)$ is measurable. Moreover, if $a(t,x)$ satisfies
	$$\|a(t,x)\|_m\leq\gamma(t)+C\|x\|_m^{q/p},$$
	for some $\gamma\in L^p(0,T)$, with $p,q\in[1,\infty)$, then $\mathcal{N}_a$ is a bounded continuous mapping from {$L^q(0,T;\mathbb{R}^m)$} into $L^p(0,T;\mathbb{R}^m)$.
\end{proposition}

In what follows, we adopt a structured approach. First, we establish a criterion to transform the fractional Cauchy problem \eqref{EDOF} into an equivalent integral formulation. Next, we prove the existence and uniqueness of a local solution and discuss how such solutions can be extended. Finally, we present a theorem concerning the existence and uniqueness of a global solution.

We begin by introducing two definitions for solutions of the Cauchy problem \eqref{EDOF}.

\begin{definition}
	\begin{itemize}
    \item[(i)] Let $\tau\in (0,T)$ and  $\phi:[0,\tau]\to\mathbb{R}^{m}$ a continuous function. Then, if $\phi$ satisfies both equations of  \eqref{EDOF} in $[0,\tau]$, we say that it is a local solution of \eqref{EDOF} in $[0,\tau]$.
	\item[(ii)] Let $\phi:[0,T]\to\mathbb{R}^{m}$ a continuous function. If $\phi$ satisfies both equations of \eqref{EDOF} in $[0,T]$, we say that it is a global solution of \eqref{EDOF}\ in $[0,T]$.
    \end{itemize}
\end{definition}

First, we focus on establishing the existence of a local solution to problem (\ref{EDOF}). To achieve this, we begin by proving its equivalent integral formulation.

{\begin{proposition}\label{202501291016}
		Let $\tau\in (0,T)$, $p\in(1,\infty)$, $\alpha\in\left(1/p,1\right)$, $\phi:[0,\tau]\to \mathbb{R}^{m}$ a continuous function and assume that $G:[0,T]\times\mathbb{R}^m\to\mathbb{R}^{m}$ is a Carath\'{e}odory map that satisfies
		\begin{equation}\label{concara}\|G(t,x)\|_m\leq\gamma(t)+C\|x\|_m^{q/p},\end{equation}
		for some $\gamma\in L^p(0,T)$, with $q\in[1,\infty)$. Then $\phi$ is a local  solution of \eqref{EDOF} in $[0,\tau]$, if and only if, $\phi$ satisfies the integral equation
		\begin{equation}\label{202505221115}
			\phi(t)=u_{0 }+\dfrac{1}{\Gamma(\alpha)}\int\limits_{0}^{t}(t-s)^{\alpha-1}G(s,\phi(s))\; ds,\;\; \forall t\in\;[0,\tau].
		\end{equation}
	\end{proposition}
	\begin{proof}  Since $\phi$ is continuous in $[0,\tau]$, Proposition \ref{carat} ensures that $G(\cdot,\phi(\cdot))\in L^p(0,T;\mathbb{R}^m)$. 

		Now, if we assume that $\phi$ is a local solution of (\ref{EDOF}) in $[0,\tau]$, since
		$$cD_{t}^{\alpha}\phi(t)=G(t,\phi(t)),\textrm{ for a.e. }t\in[0,\tau],$$
		we can apply $J_{t}^{\alpha}$ to both sides of the above equality, which, together with Proposition \ref{Pro1} [item $(v)$], allows us to conclude that
		\begin{equation}\label{202505221033}
			\phi(t)=u_{0}+\dfrac{1}{\Gamma(\alpha)}\int\limits_{0}^{t}(t-s)^{\alpha-1}G(s,\phi(s))\; ds,\textrm{ for a.e. }t\in[0,\tau].
		\end{equation}
        Finally, \cite[Theorem 7]{CarFe1} guarantees that the right-hand side of \eqref{202505221033} belongs to $C([0, \tau]; \mathbb{R}^n)$, which in turn allows us to conclude that
		\begin{equation*}
			\phi(t)=u_{0}+\dfrac{1}{\Gamma(\alpha)}\int\limits_{0}^{t}(t-s)^{\alpha-1}G(s,\phi(s))\; ds,\;\; \forall t\in\;[0,\tau].
		\end{equation*}
        For a more detailed discussion of the above identity holding for every $t \in [0, \tau]$, see Remark \ref{hardy}.

        Conversely, since $G(\cdot, \phi(\cdot)) \in L^p(0, T; \mathbb{R}^m)$, we can apply \cite[Theorem 7]{CarFe1} to obtain $J_t^\alpha G(\cdot, \phi(\cdot)) \in C([0, \tau]; \mathbb{R}^n)$. Moreover, as $\phi \in C([0, \tau]; \mathbb{R}^n)$ and satisfies the integral equation
        \begin{equation}\label{202505221036}
        \phi(t) = u_0 + \frac{1}{\Gamma(\alpha)} \int_0^t (t - s)^{\alpha - 1} G(s, \phi(s))\; ds,\quad \forall t \in [0, \tau],
        \end{equation}
        the continuity of the Riemann--Liouville fractional integral ensures that $\phi(0) = u_0$. In addition, item (iv) of Proposition \ref{Pro1} allows us to apply the operator $cD_t^\alpha$ to both sides of \eqref{202505221036}, yielding
		\begin{equation*}
			cD_{t}^{\alpha}\phi(t)=G(t,\phi(t)),\;\;\text{a.e.}\; \text{in}\; [0,\tau].
		\end{equation*}
\end{proof}}

\begin{remark}\label{hardy}
The hypotheses imposed on $p$ and $\alpha$ in Proposition \ref{202501291016} can be justified by the following observations:

\begin{itemize}
    \item[(i)] To establish the fixed-point argument used in the proof of Carath\'{e}odory's classical theorem, it is necessary that $G(\cdot, \phi(\cdot)) \in L^p(0, T; \mathbb{R}^m)$ for some $1 \leq p < \infty$, for all continuous functions $\phi:[0, T]\rightarrow\mathbb{R}^m$. This condition ensures that the integral formulation of the Cauchy problem is well-defined as a mapping between subspaces of continuous functions. This property is closely tied to the fact that, for all $v \in L^1(0,T)$, the function
    $$[0,T] \ni t \mapsto \int_0^t v(s)\,ds$$
    defines a continuous function. However, this assertion does not hold in general when dealing with the Riemann-Liouville fractional integral of order $\alpha \in (0,1)$. For example, if we consider $0<\gamma < \beta < 1$ and take $v(t) = t^{-\beta}$, we see that $v \in L^1(0,T)$ but
    $$J_t^{\gamma} v(t) = \left[\dfrac{\Gamma(1-\beta)}{\Gamma(1+\gamma-\beta)}\right]t^{\gamma-\beta},$$
    which does not define a continuous function on $[0,T]$. \vspace*{0.2cm}

    \item[(ii)] In 1928, Hardy and Littlewood established that for all functions in $L^{p}(0,T)$ with $1 < p < \infty$, the fractional integral of order $\alpha \in (1/p,1)$ is Holder continuous of order $\alpha-(1/p)$. Recently, Carvalho-Neto and Fehlberg Júnior improved this result for Bochner-Lebesgue spaces (see \cite[Theorems 7 and 9]{CarFe1}). Consequently, to ensure that the integral formulation of the Cauchy problem related to the Caputo fractional derivative of order $\alpha \in (0,1)$ is well-defined as a mapping between subspaces of continuous functions, it is essential that $G(\cdot, \phi(\cdot)) \in L^p(0, T; \mathbb{R}^m)$ for some $p \in (1/\alpha,\infty)$, for all continuous functions $\phi: [0, T] \rightarrow \mathbb{R}^m$.
\end{itemize}
\end{remark}

Now we present our first main theorem, which discusses the existence and uniqueness of local solutions to problem \eqref{EDOF}.

{\begin{theorem}\label{SNS}
Let $p \in (1, \infty)$ and $\alpha \in \left(1/p, 1\right)$. Assume that $G : [0, T] \times \mathbb{R}^m \rightarrow \mathbb{R}^m$ is a Carathéodory function that satisfies \eqref{concara} for some $\gamma \in L^p(0, T)$, with $q \in [1, \infty)$, and is locally Lipschitz continuous in the second variable for almost every $t \in [0, T]$; that is, for each fixed $u_0 \in \mathbb{R}^m$, there exist $r = r(u_0) > 0$ and $L = L(u_0) \geq 0$ such that
\begin{equation*}
    \|G(t, x) - G(t, y)\|_m \leq L \|x - y\|_m,
\end{equation*}
for all $x, y \in B_r(u_0)$ and almost every $t \in [0, T]$, where $B_r(u_0)$ denotes the open ball centered at $u_0$ with radius $r$. Then, there exists $\tau \in [0, T]$ such that \eqref{EDOF} admits a unique local solution on $[0, \tau]$.
\end{theorem}
	\begin{proof}
		Given $u_{0}\in\mathbb{R}^{m}$, let $r>0$ and $L\geq 0$ denote the constants associated with the local Lipschitz continuity of $G$ at $u_{0}$. Choose $\beta\in(0,{{r}})$ and $\tau\in(0,T)$ such that
		\begin{equation*}
			\dfrac{\tau^{\alpha}L}{\Gamma(\alpha+1)}\leq \dfrac{1}{2}\;\;\; \textrm{ and } \;\;\; \dfrac{1}{\Gamma(\alpha)}\left(\dfrac{ p-1}{\alpha p-1}\right)^{1-(1/p)}\tau^{\alpha-(1/p)}\|G(\cdot,u_{0})\|_{L^{p}(0,T;\mathbb{R}^{m})}\leq \dfrac{\beta}{2}.
		\end{equation*}
		Consider
		\begin{equation*}
			K:=\{\varphi\in C([0,\tau],\mathbb{R}^{m}):\varphi(0)=u_{0}\,\textrm{ and }\,\|\varphi(t)-u_{0}\|_{m}\leq \beta,\;\;\;\; \forall t\in[0,\tau]\},
		\end{equation*}
		a nonempty and closed subset of $C([0,\tau],\mathbb{R}^{m})$. Due to Proposition \ref{carat} and item $(ii)$ of Remark \ref{hardy}, we can define the operator  $T:K\rightarrow C([0,\tau],\mathbb{R}^{m})$  by
		\begin{equation*}
			T\varphi(t)=u_{0}+\dfrac{1}{\Gamma(\alpha)}\int_{0}^{t}(t-s)^{\alpha-1}G(s,\varphi(s))\; ds.
		\end{equation*}
		To apply the Banach fixed-point theorem, we shall prove that $T(K)\subset K$ and that $T$ is a contraction. First, observe that for $\psi\in K$, as argued before, $J_t^\alpha G(t,\psi(t))$ is continuous, which ensures that $T\psi\in C([0,\tau];\mathbb{R}^{m})$ and $T\psi(0)=u_{0}$. Moreover, notice that 
		\begin{equation*}
			\|T\psi(t)-u_{0}\|_{m}\\\leq\dfrac{1}{\Gamma(\alpha)} \int_{0}^{t}(t-s)^{\alpha-1}\big[\|G(s,\psi(s))-G(s,u_{0})\|_{m}+\|G(s,u_{0})\|_{m}\big]\;\; ds,
		\end{equation*}
		for all $t\in[0,\tau]$. Therefore, since Proposition \ref{carat} ensures that $G(\cdot,u_{0})$ belongs to $ L^{p}(0,T;\mathbb{R}^{m})$, we have 
		\begin{multline*}
			\dfrac{1}{\Gamma(\alpha)}\int_{0}^{t} (t-s)^{\alpha-1}\|G(s,u_{0})\|_{m}\;\;ds \leq\dfrac{1}{\Gamma(\alpha)}\|(t-\cdot)^{\alpha-1}\|_{L^{p/(p-1)}(0,t)}\;\;\|G(\cdot,u_{0})\|_{L^{p}(0,t;\mathbb{R}^{m})}
			\\\leq\dfrac{1}{\Gamma(\alpha)}\left(\dfrac{ p-1}{\alpha p-1}\right)^{1-(1/p)}\tau^{\alpha-(1/p)}\|G(\cdot,u_{0})\|_{L^{p}(0,T;\mathbb{R}^{m})},
		\end{multline*}
		for all $t\in[0,\tau]$, what implies that
		\begin{multline*}
			\|T\psi(t)-u_{0}\|_{m}\leq \dfrac{\tau^{\alpha}L\beta}{ \Gamma(\alpha+1)}+\dfrac{1}{\Gamma(\alpha)}\left(\dfrac{ p-1}{\alpha p-1}\right)^{1-(1/p)}\tau^{\alpha-(1/p)}\|G(\cdot,u_{0})\|_{L^{p}(0,T;\mathbb{R}^{m})}
			\leq\dfrac{\beta}{2}+\dfrac{\beta}{2}=\beta.
		\end{multline*}
	
		 Now, if $\psi_1,\psi_2 \in K$, we have that
		\begin{multline*}
			\|T\psi_1(t)-T\psi_2(t)\|_{m}\leq\dfrac{1}{\Gamma(\alpha)} \int_{0}^{t} (t-s)^{\alpha-1}\|G(s,\psi_1(s))-G(s,\psi_2(s))\|_{m}\;\; ds \\
			\leq \dfrac{\tau^{\alpha}L}{\Gamma(\alpha+1)}\sup\limits_{s\in[0,\tau]}\|\psi_1(s)-\psi_2(s)\|_{m}
			\leq\dfrac{1}{2}\sup\limits_{s\in[0,\tau]}\|\psi_1(s)-\psi_2(s)\|_{m},
		\end{multline*}
		for all $t\in[0,\tau]$, or in other words, $T$ is a contraction. Summarizing the conclusions above, and applying the Banach fixed-point theorem, we conclude that there exists $\tau \in (0, T)$ and a unique local solution $\phi : [0, \tau] \to \mathbb{R}^m$ to problem~\eqref{EDOF} in $K$. 

To complete the proof, we note that any continuous function from $[0, \tau]$ into $\mathbb{R}^n$ that also satisfies the integral formulation \eqref{202505221115} must coincide with the unique local solution $\phi \in K$ whose existence we have just established. However, we omit the details, as the result follows from a standard argument based on a classical fractional Gronwall inequality, which can be found for instance in \cite[Lemma 7.1.2]{HenG}.
\end{proof}}
Next, our attention turns to extending the local solution, which exist due to the conclusion of Theorem \ref{SNS}. To accomplish this, we initially introduce the concept of continuation of local solutions, followed by the results that prove the existence and uniqueness of such extensions.
\begin{definition}
	\begin{itemize}
    \item[(i)] Let $\phi: [0, \tau] \to \mathbb{R}^m$ be a local solution on $[0, \tau]$ of \eqref{EDOF}. If $\tau^* > \tau$ and $\phi^*: [0, \tau^*] \to \mathbb{R}^m$ is a local solution of \eqref{EDOF} on $[0, \tau^*]$, then we say that $\phi^*$ is a continuation of $\phi$ on $[0, \tau^*]$.
    \item[(ii)] We say that a function $\phi:[0,\tau^{*})\to\mathbb{R}^{m}$ is a local solution of \eqref{EDOF} on $[0,\tau^{*})$, if it is a local solution of \eqref{EDOF} on $[0,\tau]$ for all $\tau\in(0,\tau^{*})$.
    \item[(iii)] If $\phi:[0,\tau^{*})\to\mathbb{R}^{m}$ is a local solution of \eqref{EDOF} on $[0,\tau^{*})$ but is not a local solution of \eqref{EDOF} in $[0,\tau^*]$, then we call it a maximal local solution.
    \end{itemize}
\end{definition}

Before addressing our next Theorem, we establish a simple technical lemma.

\begin{lemma}\label{auxlema} Let ${t}_2>{t}_1>0$, $p\in(1,\infty)$, $\alpha\in\left(1/p,1\right)$ and $v\in L^p(0,{t}_1;\mathbb{R}^m)$. Then we have that
\begin{multline*}\left\|\dfrac{1}{\Gamma(\alpha)}\int_{0}^{{t}_1}\big[({t}_1-s)^{\alpha-1}-({t}_2-s)^{\alpha-1}\big]v(s)ds\right\|_{m}
\\\leq \dfrac{1}{\Gamma(\alpha)}\left(\dfrac{p-1}{\alpha p-1}\right)^{1-(1/p)}({t}_2-{t}_1)^{\alpha-(1/p)}\|v\|_{L^p(0,{t}_1;\mathbb{R}^m)}.
\end{multline*}
\end{lemma}
\begin{proof} Since for $0\leq a\leq b$ and $\eta\in[1,\infty)$ it holds that $(b-a)^\eta\leq b^\eta-a^\eta$, we may deduce that
\begin{multline}\label{stokes2501281515}\left\|\int_{0}^{{t}_1}\big[({t}_1-s)^{\alpha-1}-({t}_2-s)^{\alpha-1}\big]v(s)ds\right\|_{m}
\\\leq\left\{\int_{0}^{{t}_1}\big[({t}_1-s)^{(\alpha-1)p/(p-1)}-({t}_2-s)^{(\alpha-1)p/(p-1)}\big]ds\right\}^{(p-1)/p}\|v\|_{L^p(0,{t}_1;\mathbb{R}^m)}
\\=\left\{\int_{0}^{{t}_1}\left[-\int_{t_1}^{t_2}\dfrac{d}{dw}(w-s)^{(\alpha-1)p/(p-1)}\,dw\right]ds\right\}^{(p-1)/p}\|v\|_{L^p(0,{t}_1;\mathbb{R}^m)}.
\end{multline}

As a result, by applying Fubini-Tonelli's theorem to \eqref{stokes2501281515}, we arrive at
\begin{multline*}\left\|\dfrac{1}{\Gamma(\alpha)}\int_{0}^{{t}_1}\big[({t}_1-s)^{\alpha-1}-({t}_2-s)^{\alpha-1}\big]v(s)ds\right\|_{m}
\\\leq\dfrac{1}{\Gamma(\alpha)}\left\{\int_{t_1}^{t_2}\Big[(w-t_1)^{(\alpha-1)p/(p-1)}-w^{(\alpha-1)p/(p-1)}\Big]\,dw\right\}^{(p-1)/p}\|v\|_{L^p(0,{t}_1;\mathbb{R}^m)}
\\\leq\dfrac{1}{\Gamma(\alpha)}\left(\dfrac{p-1}{\alpha p-1}\right)^{1-(1/p)}({t}_2-{t}_1)^{\alpha-(1/p)}\|v\|_{L^p(0,{t}_1;\mathbb{R}^m)}.\end{multline*}
\end{proof}

{\begin{theorem}\label{ESNS}Let $p\in(1,\infty)$, $\alpha \in \left(1/p, 1\right)$, and $G: [0,T] \times \mathbb{R}^m \rightarrow \mathbb{R}^m$ be a Carathéodory map that satisfies \eqref{concara} for some $\gamma \in L^p(0,T)$, with $q\in[1,\infty)$, and is locally Lipschitz continuous in the second variable for almost every $t \in [0, T]$. If $\phi: [0, \tau] \to \mathbb{R}^m$ is the unique local solution of $(\ref{EDOF})$ on $[0, \tau]$, for some $\tau \in (0,T)$, then there exists a unique continuation $\phi^{*}$ of $\phi$ on some interval $[0, \tau + \tau_1]\subset[0,T)$, for some $\tau_1>0$.
	\end{theorem}

	\begin{proof}
		Let $\phi : [0, \tau] \to \mathbb{R}^{m}$ be the unique local solution of~\eqref{EDOF} on $[0, \tau]$. Let $B_r(\phi(\tau))$ denote the open ball centered at $\phi(\tau)$ with radius $r > 0$, and let $L \geq 0$ be the Lipschitz constant of $G$ associated with its local Lipschitz continuity on $B_r(\phi(\tau))$.
 Fix $\beta\in(0,r)$ and choose $\tau_{1}>0$ with $\tau+\tau_{1}\in(0,T)$, such that
		\begin{gather*}
			\dfrac{\tau_{1}^{\alpha}L}{\Gamma(\alpha+1)}\leq \dfrac{1}{3}\;\;\;\;\; \textrm{ and } \;\;\;\;\;
			\dfrac{1}{\Gamma(\alpha)}\left(\dfrac{ p-1}{\alpha p-1}\right)^{1-(1/p)}\tau_1^{\alpha-(1/p)}D\leq \dfrac{\beta}{3}, 
		\end{gather*}
		where
		{\begin{eqnarray}
		D=\max\big\{\|G(\cdot,\phi(\tau))\|_{L^{p}(0,T;\mathbb{R}^{m})},
			\|G(\cdot,\phi(\cdot))\|_{L^{p}(0,\tau;\mathbb{R}^{m})}\big\}.\nonumber
		\end{eqnarray}}
		Consider
		\begin{equation*}
			K:=\left\{\varphi\in C([0,\tau+\tau_{1}],\mathbb{R}^{m}):
			\begin{array}{c}
				\varphi(t)=\phi(t),\; \forall t\in[0,\tau]\;\; \textrm{ and } \\
				\|\varphi(t)-\phi(\tau)\|_{m}\leq \beta,\;\forall t\in[\tau,\tau+\tau_{1}]
			\end{array}\right\}
		\end{equation*}
		a nonempty and closed subset of $C([0,\tau+\tau_{1}],\mathbb{R}^{m})$ and the operator $T:K\to C([0,\tau+\tau_{1}];\mathbb{R}^{m})$ given by
		\begin{equation*}
			T\varphi(t)=u_{0}+\dfrac{1}{\Gamma(\alpha)}\int_{0}^{t}(t-s)^{\alpha-1}G(s,\varphi(s))ds.	
		\end{equation*}
		Let us prove that $T(K)\subset K$.
		\begin{itemize}
			\item[(i)] If $\psi\in K$, then $\psi(t)=\phi(t)$ in $[0,\tau]$ with $\phi$ the local solution of (\ref{EDOF}) in $[0,\tau]$. So, if $t\in[0,\tau]$
			$$T\psi(t)=u_{0}+\dfrac{1}{\Gamma(\alpha)}\int_{0}^{t}(t-s)^{\alpha-1}G(s,\phi(s))ds=T\phi(t)=\phi(t).$$
			\item[(ii)] If $\psi\in K$ and $t\in[\tau,\tau+\tau_{1}]$, then
			\begin{multline*}
				\|T\psi(t)-\phi(\tau)\|_{m}\leq \left\|\dfrac{1}{\Gamma(\alpha)}\int_{\tau}^{t}(t-s)^{\alpha-1}\big[G(s,\psi(s))-G(s,\phi(\tau))\big]ds\right\|_{m}\\
				+\left\|\dfrac{1}{\Gamma(\alpha)}\int_{0}^{\tau}\big[(t-s)^{\alpha-1}-(\tau-s)^{\alpha-1}\big]G(s,\phi(s))ds\right\|_{m}
				\\+\left\|\dfrac{1}{\Gamma(\alpha)}\int_{\tau}^{t}(t-s)^{\alpha-1}G(s,\phi(\tau))ds\right\|_{m}\\=:\mathcal{I}_1(t)+\mathcal{I}_2(t)+\mathcal{I}_3(t).
			\end{multline*}

		  We now proceed to estimate each of the functions $\mathcal{I}_1(t)$, $\mathcal{I}_2(t)$, and $\mathcal{I}_3(t)$ defined above.
          \begin{itemize}
          \item[(a)] A direct estimate shows that
          $$\mathcal{I}_1(t)\leq\dfrac{L\beta \; \tau_{1}^{\alpha}}{\Gamma(\alpha+1)};\vspace*{0.1cm}$$
          \item[(b)] Proposition \ref{carat} and Lemma \ref{auxlema} ensure
          $$\mathcal{I}_2(t)\leq\dfrac{1}{\Gamma(\alpha)}\left(\dfrac{p-1}{\alpha p-1}\right)^{1-(1/p)}\tau_1^{\alpha-(1/p)}\|G(\cdot,\phi(\cdot))\|_{L^{p}(0,\tau;\mathbb{R}^{m})};\vspace*{0.1cm}$$
          \item[(c)] Again, a direct computation guarantees 
          $$\mathcal{I}_3(t)\leq\dfrac{1}{\Gamma(\alpha)}\left(\dfrac{p-1}{\alpha p-1}\right)^{1-(1/p)}\tau_1^{\alpha-(1/p)}\|G(\cdot,\phi(\tau))\|_{L^{p}(0,T;\mathbb{R}^{m})}.$$
          \end{itemize}

          Therefore 
			\begin{equation*}
				\|T\psi(t)-\phi(\tau)\|_{m}\leq\beta,\;\;\; \text{for all}\;\; t\in[\tau,\tau+\tau_{1}].
			\end{equation*}
		\end{itemize}
		Now, let us prove that $T$ is a contraction. For $\psi_1,\psi_2\in K$ and  $t\in[0,\tau+\tau_{1}]$, it holds that
		\begin{eqnarray}
			\|T\psi_1(t)-T\psi_2(t)\|_{m}
			&\leq&\dfrac{1}{\Gamma(\alpha)} \int_{\tau}^{t} (t-s)^{\alpha-1}\|G(s,\psi_1(s))-G(s,\psi_2(s))\|_{m}\;\; ds\nonumber \\
			&\leq&\dfrac{1}{3}\sup\limits_{s\in[0,\tau+\tau_{1}]}\|\psi_1(s)-\psi_2(s)\|_{m}.\nonumber
		\end{eqnarray}

      Therefore, Banach's fixed-point theorem ensures the existence of a unique fixed point $\phi^* \in K$ for the integral equation, which is, in fact, the only continuation of $\phi$ over the interval $[0, \tau + \tau_1]$. The uniqueness of this continuation can be established by applying the same argument used at the end of the proof of Theorem~\ref{SNS}.
\end{proof}}
Next, we present an important lemma, which provides support for the theorem concerning the existence and uniqueness of the global solution to problem (\ref{EDOF}).
{\begin{lemma}\label{LESNS}
		Let $p\in(1,\infty)$, $\alpha\in\left(1/p,1\right)$, $\tau\in(0,T]$, and $G:[0,T]\times\mathbb{R}^m\rightarrow \mathbb{R}^m$  a Carath\'{e}odory map that satisfies \eqref{concara} for some $\gamma\in L^p(
		0,T)$, with $q\in[1,\infty)$. If $\phi:[0,\tau)\longrightarrow\mathbb{R}^{m}$ is a local solution of \eqref{EDOF} and belongs to $L^q(0,\tau;\mathbb{R}^m)$, then we can extend $\phi$ continuously on $[0, \tau]$ in a unique manner. Furthermore, this extension is a solution of \eqref{EDOF} on the interval $[0, \tau]$.
	\end{lemma}
	\begin{proof}
If $\phi$ is a solution of (\ref{EDOF}) on the interval $[0,\tau)$, for $t_1,t_2 \in [0,\tau)$, where $t_2 > t_1 > 0$, we have
\begin{multline*}
    \|\phi(t_1)-\phi(t_2)\|_{m} \leq \dfrac{1}{\Gamma(\alpha)} \int_{t_1}^{t_2} (t_2-s)^{\alpha-1} \|G(s,\phi(s))\|_m\,ds \\
    + \int_{0}^{t_1} \big[(t_1-s)^{\alpha-1}-(t_2-s)^{\alpha-1}\big] \|G(s,\phi(s))\|_m\,ds.
\end{multline*}
Therefore, by Proposition \ref{carat} and Lemma \ref{auxlema}, we conclude that
$$
			\|\phi(t_1)-\phi(t_2)\|_{m}\leq \dfrac{2}{\Gamma(\alpha)}\left(\dfrac{p-1}{\alpha p-1}\right)^{1-(1/p)}(t_2-t_1)^{\alpha-(1/p)}\|G(\cdot,\phi(\cdot))\|_{L^{p}(0,\tau;\mathbb{R}^{m})}.
$$

This proves that $\phi$ is uniformly continuous in $[0,\tau)$. Therefore we may extend it uniquely to $[0,\tau]$ and conclude that it is a solution of \eqref{EDOF} on the interval $[0, \tau]$.
\end{proof}}
%
{\begin{theorem}\label{stokes202501290918}
		Let $p\in(1,\infty)$, $\alpha\in\left(1/p,1\right)$, $G:[0,T]\times\mathbb{R}^m\rightarrow \mathbb{R}^m$ a Carath\'{e}odory map that satisfies \eqref{concara} for some $\gamma\in L^p(0,T)$, with $q=p$ and is locally Lipschitz continuous in the second variable for almost every $t \in [0, T]$. Then, problem \eqref{EDOF} has a unique global solution in $[0,T]$.
	\end{theorem}
	\begin{proof}
		Consider 
		\begin{equation*}
			H:=\big\{\tau\in(0,T]:\textrm{exists a unique }\phi_{\tau}:[0,\tau]\to \mathbb{R}^n
				\textrm{ that is a local solution of (\ref{EDOF}) on }[0,\tau] 
			\big\}.
		\end{equation*}
		Note that Theorem \ref{SNS} guarantees that $H$ is nonempty. Therefore, if we denote $\tau^* = \sup H$, we can define a continuous function $\phi:[0,\tau^*)\to \mathbb{R}^n$, which is the unique solution of problem \eqref{EDOF} in $[0,\tau]$, for all $\tau\in(0,\tau^*)$. Now notice that due to \eqref{concara} we have
\begin{multline*}\|\phi(t)\|_m\leq \underbrace{\|u_0\|_m+\dfrac{1}{\Gamma(\alpha)}\left(\dfrac{p-1}{\alpha p-1}\right)^{1-(1/p)}T^{\alpha-(1/p)}\|\gamma\|_{L^p(0,T)}}_{:=M}+\dfrac{C}{\Gamma(\alpha)}\int_{0}^{t}(t-s)^{\alpha-1}\big\|\phi(s)\|_m \,ds,\end{multline*}
for all $t\in[0,\tau^*)$. But then, since $\phi\in L^1_{loc}(0,\tau^*;\mathbb{R}^m)$, the fractional version of Gronwall's inequality (cf. \cite[Lemma 7.1.2]{HenG}) ensures that 
$$\|\phi(t)\|_m\leq ME_\alpha(C^{1/\alpha} t),$$ 
for all $t \in [0, \tau^*)$, where $E_\alpha(z)$ denotes the Mittag-Leffler function (see \cite{GoKiMaRo1} for more details on the one-parameter Mittag-Leffler functions). From this, we conclude that $\phi \in L^\infty(0, \tau^*; \mathbb{R}^m)$.

Consequently, we can deduce that $\tau^* = T$. Indeed, if this were not the case, that is, if $\tau^* < T$, then Theorem~\ref{ESNS} and Lemma~\ref{LESNS} would ensure the existence and uniqueness of a solution defined on an interval properly containing $[0, \tau^*]$, contradicting the definition of $\tau^*$. Therefore, since $\tau^* = T$, we may apply Lemma~\ref{LESNS} once more to ensure the existence of a unique extension of $\phi$ on $[0, T]$, as desired.
\end{proof}}

\begin{remark}
It is worth noting that Theorem \ref{stokes202501290918} becomes significantly more challenging to prove when $q \neq p$, as extending Gronwall's inequality to this scenario results in a far more complex task. Since such an improvement does not contribute to our objectives, we have opted not to address this generalization here.
\end{remark}

\section{The $L^p_\alpha(0,T;X)$ Spaces}
\label{fracspace}

As mentioned in the introduction, it is essential to discuss a new subspace of the Bochner--Lebesgue spaces, which plays a fundamental role in the main results of this work. Therefore, to keep the setting as general as possible, we assume throughout this section that $X$ is a Banach space and $T > 0$.

\begin{definition} \label{202503061950}
For $\alpha \in (0,1]$ and $p \in [1, \infty)$, the space of $(\alpha, p)$-Bochner integrable functions is denoted by $L^p_\alpha(0, T; X)$ and consists of all Bochner measurable functions $v : [0, T] \to X$ such that
$$
\int_0^T \frac{(T - s)^{\alpha - 1}}{\Gamma(\alpha)} \|v(s)\|_X^p \, ds < \infty.
$$
For completeness, we also define $L^p_0(0, T; X) := L^\infty(0, T; X)$ for any $p \in [1, \infty)$.
\end{definition}

\begin{remark} 
\begin{itemize}
\item[(i)] It is interesting to note that the choice of the interval $[0,T]$, in the definition of the sets $L^p_\alpha(0,T;X)$, was made solely for the purposes of the studies developed in this manuscript. Nevertheless, replacing $[0,T]$ with any other interval $[t_0,t_1]$ does not affect the results obtained in this section, except for slight changes in the constants involved in the estimates.\vspace*{0.2cm}
\item[(ii)] Observe that this definition could be extended to $\alpha>1$. However, since it falls outside the scope of the results we are proving in this manuscript, we prefer not to discuss this case at this moment, leaving such an analysis for a future related work.
\end{itemize}
\end{remark}

Let us begin our study of $L^p_\alpha(0,T;X)$ by proving that each of these sets forms a normed vector space, which is also a subspace of $L^p(0,T;X)$.

\begin{proposition}\label{202503051659} Let $\alpha\in[0,1]$ and $p\in[1,\infty)$. Then $L^p_1(0,T;X)=L^p(0,T;X)$ and $L^p_0(0,T;X)=L^\infty(0,T;X)$, both of which are subspaces of $L^p(0,T;X)$. On the other hand, for $\alpha\in(0,1)$, $L^p_\alpha(0,T;X)$ is a vector subspace of $L^p(0,T;X)$.
	\end{proposition}
	
\begin{proof}
Assume that $\alpha\in(0,1)$. First, observe that for any $v\in L_\alpha^p(0,T;X)$, we have
\begin{equation*}\int_0^T \|v(s)\|^p\,ds=\int_0^T (T-s)^{1-\alpha}(T-s)^{\alpha-1}\|v(s)\|^p\,ds\\\leq \Gamma(\alpha)T^{1-\alpha}\|v\|_{L^p_\alpha(0,T;X)}<\infty.\end{equation*}
Thus, $L^p_\alpha(0,T;X)$ is contained in $L^p(0,T;X)$. Now, let $\lambda\in\mathbb{R}$ and $v_1,v_2 \in L^p_\alpha(0,T;X)$. Then,
\begin{multline*}
	\int_0^T \frac{(T-s)^{\alpha-1}}{\Gamma(\alpha)} \|\lambda v_1(s) + v_2(s)\|_{X}^p \, ds 
	\leq {2^{p-1}\max\{1,|\lambda|^p\}} \left[\int_0^T \frac{(T-s)^{\alpha-1}}{\Gamma(\alpha)} \|v_1(s)\|_{X}^{p} \, ds \right. \nonumber\\
	+ \left. \int_0^T \frac{(T-s)^{\alpha-1}}{\Gamma(\alpha)} \|v_2(s)\|_{X}^p \, ds\right]<\infty.\nonumber
\end{multline*}
Hence, $\lambda v_1+v_2\in L^p_\alpha(0,T;X)$, which proves that $L^p_\alpha(0,T;X)$ is a vector subspace of $L^p(0,T;X)$.
\end{proof}

Now, we define an appropriate norm for the vector space $L^p_\alpha(0,T;X)$.

{\begin{proposition} For $\alpha\in[0,1]$ and $p\in[1,\infty)$, consider $\|\cdot\|_{L^p_\alpha(0,T;X)}:L^p_\alpha(0,T;X)\rightarrow\mathbb{R}$ given by
		$$\|v\|_{L^p_\alpha(0,T;X)}:=\left\{\begin{array}{ll}\|v\|_{L^\infty(0,T;X)},&\textrm{if }\alpha=0,\vspace*{0.2cm}\\
\left(\displaystyle\int_0^T \frac{(T-s)^{\alpha-1}}{\Gamma(\alpha)} \|v(s)\|_{X}^{p}  ds\right)^{1/p},&\textrm{if }\alpha\in(0,1].\end{array}\right.$$
		With this, $\Big(L^p_\alpha(0,T;X), \|\cdot\|_{L^p_\alpha(0,T;X)}\Big)$ defines a normed vector space.
\end{proposition}
	\begin{proof} Assume that $\alpha\in(0,1)$. 
\begin{itemize}
\item[(i)] If $v \in L^p_\alpha(0,T;X)$ and $\|v\|_{L^p_\alpha(0,T;X)} = 0$, since $(T-s)^{\alpha-1} \|v(s)\|_{X}^{p} \geq 0$ for almost every $s \in [0,T]$, it follows that $\|v(s)\|_{X}^{p} = 0$ for almost every $s \in [0,T]$. Therefore $v=0$.\vspace*{0.2cm}
\item[(ii)] It is straightforward to verify that for $v \in L^p_\alpha(0,T;X)$ and $\lambda \in \mathbb{R}$, we have $\|\lambda v\|_{L^p_\alpha(0,T;X)}  = |\lambda| \|v\|_{L^p_\alpha(0,T;X)}$.\vspace*{0.2cm}
\item[(iii)] For $v_1, v_2 \in L^p_\alpha(0,T;X)$ we have
		\begin{multline*}
			\qquad J^\alpha_t \|v_1(t)+v_2(t)\|_{X}^p\Big|_{t=T} \\
			= \frac{1}{\Gamma(\alpha)} \int_0^T (T-s)^{(\alpha-1)/p} \|v_1(s)+v_2(s)\|_{X} (T-s)^{(\alpha-1)(p-1)/p} \|v_1(s)+v_2(s)\|_{X}^{p-1} ds,
		\end{multline*}
        what ensures that
        \begin{multline*}
		\qquad J^\alpha_t \|v_1(t)+v_2(t)\|_{X}^p\Big|_{t=T} \\\leq
        \int_0^T \left[ \frac{(T-s)^{\alpha-1}}{\Gamma(\alpha)} \right]^{1/p} \|v_1(s)\|_{X} \left[ \frac{(T-s)^{\alpha-1}}{\Gamma(\alpha)} \right]^{\frac{p-1}{p}} \|v_1(s)+v_2(s)\|_{X}^{p-1} ds\\
		+\int_0^T \left[ \frac{(T-s)^{\alpha-1}}{\Gamma(\alpha)} \right]^{1/p} \|v_2(s)\|_{X} \left[ \frac{(T-s)^{\alpha-1}}{\Gamma(\alpha)} \right]^{\frac{p-1}{p}} \|v_1(s)+v_2(s)\|_{X}^{p-1} ds.
		\end{multline*}
        Thus, Hölder's inequality allows us to achieve that
		\begin{multline*}
		\qquad	J^\alpha_t \|v_1(t)+v_2(t)\|_{X}^p\Big|_{t=T} \\\leq  \left( \int_0^T \frac{(T-s)^{\alpha-1}}{\Gamma(\alpha)} \|v_1(s)\|_{X}^{p}  ds \right)^{\frac{1}{p}}  \left( \int_0^T \frac{(T-s)^{\alpha-1}}{\Gamma(\alpha)} \|v_1(s)+v_2(s)\|_{X}^p ds \right)^{\frac{p-1}{p}}
			\\+ \left( \int_0^T \frac{(T-s)^{\alpha-1}}{\Gamma(\alpha)} \|v_2(s)\|_{X}^p ds \right)^{\frac{1}{p}}  \left( \int_0^T \frac{(T-s)^{\alpha-1}}{\Gamma(\alpha)} \|v_1(s)+v_2(s)\|_{X}^p ds \right)^{\frac{p-1}{p}}.
		\end{multline*}
        Therefore
		\begin{equation} \label{202503051650}
			\| v_1+v_2 \|_{L^p_\alpha(0,T;X)}^p \leq \big [ \|v_1 \|_{L^p_\alpha(0,T;X)} + \|v_2 \|_{L^p_\alpha(0,T;X)} \big ] \| v_1+v_2 \|^{p-1}_{L^p_\alpha(0,T;X)}.
		\end{equation}
		It follows now directly from \eqref{202503051650} that
		\begin{equation*}
			\| v_1+v_2 \|_{L^p_\alpha(0,T;X)} \leq \big  \|v_1 \|_{L^p_\alpha(0,T;X)} + \|v_2 \|_{L^p_\alpha(0,T;X)},
		\end{equation*}
		completing the proof that $\| \cdot \|_{L^p_\alpha(0,T;X)}$ defines a norm in $L^p_\alpha(0,T;X)$.
\end{itemize}
\end{proof}}

Now, in order to establish some relationships between $L^p_\alpha(0,T;X)$ for different values of $\alpha\in[0,1]$ and $p\in[1,\infty)$, as well as its connection with the classical $L^p(0,T;X)$, we present the following series of results.

\begin{proposition}\label{202501291512} Let $\alpha,\beta\in[0,1]$, with $\alpha<\beta$, and $p\in[1,\infty)$. Then, $L^p_\alpha(0,T;X) \subsetneq L^p_\beta(0,T;X)$, and for all $f\in L^p_\alpha(0,T;X)$ it holds that
\begin{equation}\label{202503051621}
\|v\|_{L^p_\beta(0,T;X)}\leq\left\{\begin{array}{ll}\left(\dfrac{T^{\beta}}{\Gamma(\beta+1)}\right)^{1/p}\|v\|_{L^p_\alpha(0,T;X)},&\textrm{ for }\alpha=0,\vspace*{0.3cm}\\
\left(\dfrac{T^{\beta-\alpha}\Gamma(\alpha)}{\Gamma(\beta)}\right)^{1/p}\|v\|_{L^p_\alpha(0,T;X)},&\textrm{ for }\alpha\in(0,1].\end{array}\right.
\end{equation}
\end{proposition}

\begin{proof} First, let $\alpha = 0$. If $v\in L^p_0(0,T;X)$, it follows from \cite[Theorem 7]{CarFe1} that $J_t^\beta \|v(t)\|_X^p \in C([0,T];\mathbb{R})$. Consequently, we obtain $L^p_0(0,T;X) \subset L^p_\beta(0,T;X)$ for all $\beta\in(0,1]$, and prove the desired inequality. To show that this inclusion is strict, consider $v(t) = t^{-\beta/p}$, which belongs to $L^p_\beta(0,T;\mathbb{R})$ but not to $L^p_0(0,T;\mathbb{R})$.

Now, assume that $0< \alpha < \beta \leq 1$. For any $v\in L^p_\alpha(0,T;X)$, we have
\begin{multline*}
\int_0^T(T-s)^{\beta-1}\|v(s)\|_X^p\,ds = \int_0^T(T-s)^{\beta-\alpha} (T-s)^{\alpha-1}\|v(s)\|_X^p\,ds \\
\leq T^{\beta-\alpha}\int_0^T(T-s)^{\alpha-1}\|v(s)\|_X^p\,ds.
\end{multline*}
This implies the norm bound \eqref{202503051621}, which establishes the inclusion $L^p_\alpha(0,T;X) \subset L^p_\beta(0,T;X)$ for $0< \alpha < \beta \leq 1$. 

Finally, to verify that this inclusion is strict, consider $v(t)=(T-t)^{-\gamma/p}$ for any $\gamma\in(\alpha,\beta)$. In this case, we have that $v\in L^p_\beta(0,T;\mathbb{R})$ but $v\not\in L^p_\alpha(0,T;\mathbb{R})$.

\end{proof}

\begin{proposition}
		
		Let $\alpha\in(0,1]$ and $p,q\in[1,\infty)$, with $p< q$. Then $L^q_\alpha(0,T;X) \subsetneq L_{\alpha}^{p}(0,T;X)$, and for all $v\in L^p_\alpha(0,T;X)$ it holds that 
		\begin{equation*} 
			\| v \|_{L_{\alpha}^{q}(0,T;X)} \leq \left(\frac{T^{\alpha}}{\Gamma(\alpha+1)}\right)^{\frac{q-p}{pq}} \| v \|_{L_{\alpha}^{p}(0,T;X)}.
		\end{equation*}
	\end{proposition}
	
	\begin{proof}
		Assume that $\alpha\in(0,1)$. For  $v \in L^p_\alpha(0,T;X)$, we have
		\begin{multline*}
			 \| v \|^p_{L_{\alpha}^{p}(0,T;X)}= \int_0^T  \left(\frac{(T-s)^{\alpha-1}}{\Gamma(\alpha)}\right)^{\frac{p}{q}}\|v(s)\|_{X}^{p} \left(\frac{(T-s)^{\alpha-1}}{\Gamma(\alpha)}\right)^{\frac{q-p}{q}}ds\nonumber\\
			\leq \left(\int_0^T \frac{(T-s)^{\alpha-1}}{\Gamma(\alpha)} \|v(s)\|_{X}^{q}  ds\right)^{\frac{p}{q}}\left(\int_0^T \frac{(T-s)^{\alpha-1}}{\Gamma(\alpha)} ds\right)^{\frac{q-p}{q}}\nonumber\\
			=\left(\frac{T^{\alpha}}{\Gamma(\alpha+1)}\right)^{\frac{q-p}{q}} \| v \|^p_{L_{\alpha}^{q}(0,T;X)}.\nonumber
		\end{multline*}

        Now, to prove the strictness of the inclusion, consider $\gamma\in(p,q)$ and let $v(t)=(T-t)^{-\alpha/\gamma}$. Then $v\in L^p_\alpha(0,T;\mathbb{R})$ but $v\not\in L^q_\alpha(0,T;\mathbb{R})$.

\end{proof}

\begin{theorem} For $\alpha\in(0,1)$ and $p,q\in[1,\infty)$, with $q>p/\alpha$, we have the strict inclusion $L^q(0,T;X)\subsetneq L^p_\alpha(0,T;X)$. Furthermore, for all $v\in L^q(0,T;X)$, it holds that 
\begin{equation}\label{202503061457}
\|v\|_{L^p_\alpha(0,T;X)}\leq \left(\dfrac{q-p}{\alpha q-p}\right)^{(q-p)/pq}\left(\dfrac{T^{(\alpha q-p)/pq}}{\Gamma(\alpha)^{1/p}}\right)\|v\|_{L^q(0,T;X)}.
\end{equation}
Additionally, we have
$$
  L^p_\alpha(0,T;X) \nsubseteq L^{p/\alpha}(0,T;X)
  \quad \text{and} \quad
  L^{p/\alpha}(0,T;X) \nsubseteq L^p_\alpha(0,T;X).
$$
\end{theorem}

\begin{proof} For $q>p/\alpha$ and $v\in L^q(0,T;X)$, observe that
$$\int_0^T(T-s)^{\alpha-1}\|v(s)\|^p\,ds\leq \left(\int_0^T(T-s)^{(\alpha-1)q/(q-p)}\,ds\right)^{(q-p)/q}\left(\int_0^T\|v(s)\|^{q}\,ds\right)^{p/q},$$
which implies \eqref{202503061457}. 

To verify that $L^q(0,T;X)\subsetneq L^p_\alpha(0,T;X)$ for all $q>p/\alpha$, choose $\gamma\in(1/q,\alpha/p)$ and define $v(t)=(T-t)^{-\gamma}$. Then, $v\in L^p_\alpha(0,T;\mathbb{R})$ but $v\not\in L^q(0,T;\mathbb{R})$.

Finally, to establish the last part of this theorem, we need to provide two counterexamples showing that the sets $L^p_\alpha(0,T;X)$ and $L^{p/\alpha}(0,T;X)$ are not contained in each other. 
Let us start with the simpler case. Consider $v(t) = t^{-\alpha/p}$, and observe that $v \in L^p_\alpha(0,T;\mathbb{R})$ while $v \notin L^{p/\alpha}(0,T;\mathbb{R})$. 
To prove the remaining non-inclusion, assume, just to simplify the computations, that $T = 1$, and consider $v:(0,1)\to\mathbb{R}$ defined by 
$$
v(t) = \dfrac{1}{(1-t)^{\alpha/p}\big[\log(e/(1-t))\big]^{1/p}}.
$$

Observe that, by applying the change of variables $u = \log(e/(1-s))$, we obtain
\begin{equation*}
  \|v\|_{L^{p/\alpha}(0,1;\mathbb{R})}^{p/\alpha}
  = \int_0^1 \left\{ \frac{1}{(1-s)^{\alpha/p} [\log(e/(1-s))]^{1/p}} \right\}^{p/\alpha} ds
  = \int_1^\infty u^{-1/\alpha} \, du < \infty,
\end{equation*}
which shows that $v \in L^{p/\alpha}(0,1;\mathbb{R})$. On the other hand, for $t \in (0,1)$ we have
\begin{equation*}
  J_t^\alpha |v(t)|^p
  = \frac{1}{\Gamma(\alpha)} \int_0^t \frac{(t-s)^{\alpha-1}}{(1-s)^{\alpha} \log(e/(1-s))} \, ds
  \geq \frac{1}{\Gamma(\alpha)} \int_0^t \frac{1}{(1-s) \log(e/(1-s))} \, ds.
\end{equation*}
By making the substitution $u = \log(e/(1-s))$, it follows that
\begin{equation*}
  J_t^\alpha |v(t)|^p
  \geq \frac{1}{\Gamma(\alpha)} \int_1^{\log(e/(1-t))} \frac{du}{u}
  = \frac{1}{\Gamma(\alpha)} \log\big(\log(e/(1-t))\big),
\end{equation*}
which diverges as $t \to 1^-$. Consequently, $v \notin L_\alpha^p(0,1;\mathbb{R})$.
\end{proof}

\begin{remark} We have that $L^p_1(0,T;X) = L^p(0,T;X)$ and $L^p_0(0,T;X) = L^\infty(0,T;X)$. Therefore, by Proposition~\ref{202501291512}, we can assert that the family of sets $L^p_\alpha(0,T;X)$ interpolates between $L^p(0,T;X)$ and $L^\infty(0,T;X)$ as $\alpha$ decreases from $1$ to $0$. 
\end{remark}

We conclude our analysis of the relationship between the spaces $L^p_\alpha(0,T;X)$ and $L^p(0,T;X)$ with the following result.

\begin{theorem} For $\alpha \in (0,1)$ and $p \in [1,\infty)$, let $q \in (p, p/\alpha)$. Then neither of the following inclusions holds:
$$
L^q(0,T;X) \subset L^p_\alpha(0,T;X)
\quad \text{nor} \quad
L^p_\alpha(0,T;X) \subset L^q(0,T;X).
$$
\end{theorem}

\begin{proof} We prove the result only in the scalar case $X = \mathbb{R}$, since if $\varphi : (0,T) \to \mathbb{R}$ satisfies the desired properties, then for any $x_0 \in X$ with $\|x_0\|_X = 1$, the function $v(t) = \varphi(t)x_0$ satisfies $\|v(t)\|_X = |\varphi(t)|$. Consequently, the same inclusion relations follow.

Let us first prove that $L^q(0,T;\mathbb{R})\not\subset L^p_\alpha(0,T;\mathbb{R})$. Since $q<p/\alpha$, there exists $\gamma$ such that
\begin{equation}\label{202510301737}
\frac{\alpha}{p}\le \gamma<\frac{1}{q}.
\end{equation}
Consider $\varphi_1: (0,T) \rightarrow \mathbb{R}$ defined by $\varphi_1(s) = (T - s)^{-\gamma}\,\chi_{(T/2,\,T)}(s)$, where $\chi_{(T/2,\,T)}$ denotes the characteristic function of the interval $(T/2, T)$. 

First, note that by the change of variables $u = T - s$,
$$
\int_{T/2}^T |\varphi_1(s)|^q\,ds
= \int_{0}^{T/2} u^{-\gamma q}\,du
< \infty
\iff -\gamma q > -1
\iff \gamma q < 1,
$$
which holds by \eqref{202510301737}.

Next, we show that $\varphi_1 \notin L^p_\alpha(0,T;\mathbb{R})$.  
Applying again the change of variables $u = T - s$, we obtain
$$
\int_{T/2}^T \frac{(T - s)^{\alpha - 1}}{\Gamma(\alpha)} |\varphi_1(s)|^p\,ds
= \frac{1}{\Gamma(\alpha)} \int_0^{T/2} u^{\alpha - 1 - \gamma p}\,du.
$$
It follows directly that the integral on the right-hand side diverges whenever $\alpha - 1 - \gamma p \le -1$, that is, if $\alpha - \gamma p \le 0$.  
This condition is indeed satisfied for $\gamma \ge \alpha/p$, and therefore $\varphi_1 \notin L^p_\alpha(0,T;\mathbb{R})$.\vspace*{0.2cm}

Now let us prove that $L^p_\alpha(0,T;\mathbb{R})\not\subset L^q(0,T;\mathbb{R})$. Since $q>p$, there exists $\delta$ such that
\begin{equation}\label{202510301751}
\frac{1}{q}\le \delta<\frac{1}{p}.
\end{equation}
Choose $t_0\in(0,T)$ and $\varepsilon\in\big(0,\min\{t_0,T-t_0\}\big)$. Consider function $\varphi_2 : (0,T) \rightarrow \mathbb{R}$, given by $\varphi_2(s)=|s-t_0|^{-\delta}\,\chi_{(t_0-\varepsilon,\,t_0+\varepsilon)}(s),$ where $\chi_{(t_0-\varepsilon,\,t_0+\varepsilon)}$ denotes the characteristic function of the interval $(t_0-\varepsilon,\,t_0+\varepsilon)$. 

On $(t_0-\varepsilon,t_0+\varepsilon)$ the weight $(T-s)^{\alpha-1}$ 
is bounded above and below by positive constants. Hence
$$
\|\varphi_2\|_{L^p_\alpha}^p
\leq M \int_{t_0-\varepsilon}^{t_0+\varepsilon} |s-t_0|^{-\delta p}\,ds
=2\int_0^\varepsilon u^{-\delta p}\,du<\infty
\iff \delta p<1,
$$
which holds by \eqref{202510301751}. Thus $\varphi_2\in L^p_\alpha(0,T;\mathbb{R})$.

On the other hand,
$$
\int_{0}^{T} |\varphi_2(s)|^q\,ds=\int_{t_0-\varepsilon}^{t_0+\varepsilon} |\varphi_2(s)|^q\,ds
=2\int_0^\varepsilon r^{-\delta q}\,dr
=+\infty,
\quad\text{whenever }\delta q\ge 1,
$$
and since $\delta\ge 1/q$, we have $\delta q\ge 1$. Therefore $\varphi_2\notin L^q(0,T;\mathbb{R})$.
\end{proof}

The next result is dedicated to prove that $L^{p}_\alpha(0,T;X)$ is a Banach space.

\begin{theorem} 
		Assume that $\alpha\in[0,1]$ and $p\in[1,\infty)$. Then, $L^p_\alpha(0,T;X)$, equipped with the norm $\|\cdot\|_{L^p_\alpha(0,T;X)}$, is a Banach space.
	\end{theorem}
	\begin{proof} The cases $\alpha = 0$ and $\alpha = 1$ are classical. Hence, assume that $\alpha \in (0,1)$. 
Let $(v_k)_{k \in \mathbb{N}} \subset L^p_\alpha(0,T;X)$ be a Cauchy sequence. Observe that 
		\begin{equation}
			\int_{0}^{T}(T-s)^{\alpha-1}\|v_{m}(s)-v_n(s)\|_{X}^{p}ds=\int_{0}^{T}\left\|(   T-s)^{\frac{\alpha-1}{p}}\big[v_{m}(s)-v_n(s)\big]\right\|_{X}^{p}ds.\nonumber
		\end{equation}
		Then, $\left((T-\cdot)^{\frac{\alpha-1}{p}}v_{k}(\cdot)\right)$ is a Cauchy sequence in $L^{p}(0,T;X)$. Therefore, there exists $v\in L^{p}(0,T;X)$ such that $\left((T-\cdot)^{\frac{\alpha-1}{p}}v_{k}(\cdot)\right)\to v$ in $L^{p}(0,T;X)$. 
        
        It is not difficult to verify that $(T-\cdot)^{\frac{1-\alpha}{p}}v(\cdot)\in L^{p}_\alpha(0,T;X)$, and that
		\begin{equation*}
			\int_{0}^{T}(T-s)^{\alpha-1}\left\|v_{k}(s)-(T-s)^{\frac{1-\alpha}{p}}v(s)\right\|_{X}^{p}ds= \left\|(T-s)^{\frac{\alpha-1}{p}}v_{k}-v\right\|_{L^{p}(0,T;X)}^{p},
		\end{equation*}
        what implies that $v_{k}\to (T-\cdot)^{\frac{1-\alpha}{p}}v(\cdot)$ in $L^{p}_\alpha(0,T;X)$, as we wanted.
\end{proof}

Finally, we address the reflexivity properties of the space $L_{\alpha}^{p}(0,T;X)$.

{\begin{theorem}\label{202503130949}
		For $\alpha\in(0,1]$ and $p\in(1,\infty)$, we have that $L_{\alpha}^{p}(0,T;X)$ is isometrically isomorphic to $L^p(0,T;X)$. Therefore, if $X$ is reflexive, then $L_{\alpha}^{p}(0,T;X)$ is also reflexive.
	\end{theorem}
	\begin{proof}
		To begin the first part of the proof, let us consider the set
		\begin{equation*}
			 \big(L^{p}(0,T;X)\big)_\alpha:=\left\{\left(\frac{(T-\cdot)^{\alpha-1}}{\Gamma(\alpha)}\right)^{1/p}v(\cdot):\textrm{ such that }v\in L^p_\alpha(0,T;X)\right\}.
		\end{equation*}
		Observe that $(L^{p}(0,T;X))_{\alpha}$ is a subset of $L^{p}(0,T;X)$ since, for any $v\in L_\alpha^{p}(0,T;X)$, we have
		\begin{equation}\label{202503071233}			\left\|\left(\frac{(T-\cdot)^{\alpha-1}}{\Gamma(\alpha)}\right)^{1/p}v(\cdot)\right\|_{L^{p}(0,T;X)}=\left(\int_{0}^{T}\frac{(T-s)^{\alpha-1}}{\Gamma(\alpha)}\|v(s)\|_{X}^{p}ds\right)^{1/p}=\|v\|_{L^{p}_\alpha(0,T;X)}.
		\end{equation}	
		Thus, if we define 
		\begin{equation}\label{202503121439}
			\begin{array}{c c l}
				G_p:L_{\alpha}^{p}(0,T;X)&\to& L^{p}(0,T;X)\\
				v(\cdot)&\to&\left(\dfrac{(T-\cdot)^{\alpha-1}}{\Gamma(\alpha)}\right)^{1/p}v(\cdot),
			\end{array}
		\end{equation}
		due to \eqref{202503071233}, we deduce that $G_p$ is an isometry from $L_{\alpha}^{p}(0,T;X)$ into $(L^{p}(0,T;X))_{\alpha}$. Since $L^{p}_\alpha(0,T;X)$ is a Banach space, we deduce that $G_p(L_{\alpha}^{p}(0,T;X)) = (L^{p}(0,T;X))_{\alpha}$ is a closed subset of $L^{p}(0,T;X)$.

Now observe that if $\theta\in C_c((0,T);X)$, which means that $\theta$ is a continuous function with compact support contained in $(0,T)$, then by considering  
$$\kappa(t):=\big[(T-t)^{(1-\alpha)}\Gamma(\alpha)\big]^{1/p}\theta(t),$$
we deduce that $\kappa\in L^p_\alpha(0,T;X)$. This implies that $\theta=G_p(\kappa)\in (L^{p}(0,T;X))_{\alpha}$, which means that $C_c((0,T);X)\subset (L^{p}(0,T;X))_{\alpha}$. Since $C_c((0,T);X)$ is dense in $L^p(0,T;X)$, it follows that $(L^{p}(0,T;X))_{\alpha}$ is dense in $L^p(0,T;X)$. However, since we already know that it is closed, we have proved that $L_{\alpha}^{p}(0,T;X)$ is isometrically isomorphic to $L^p(0,T;X)$.

Finally, if $X$ is reflexive, then $L^p(0,T;X)$ is reflexive, which implies, thanks to the isometric isomorphism, that $L^p_\alpha(0,T;X)$ is also reflexive.
\end{proof}}

There are several additional properties of the spaces $L^p_\alpha(0,T;X)$ that can be proved and deserve further attention. However, in order to maintain the focus on the main results and avoid unnecessary digressions, we conclude this section with one final result, which provides a complete characterization of the dual space of $L_{\alpha}^{p}(0,T;X)$. For the remainder of this paper, we recall that for a Banach space $X$, its dual will be denoted by $X'$.

\begin{lemma}\label{202503121452} Let $Z_1$ and $Z_2$ be normed vector spaces, and let $T: Z_1 \rightarrow Z_2$ be an isometric isomorphism. Then there exists an isometric isomorphism $T^\prime: Z_2^\prime \rightarrow Z_1^\prime$ given by $T^\prime(\varphi) = \varphi \circ T$. We recall that, for any normed vector space $Z$, the symbol $\langle \cdot, \cdot \rangle_{Z^\prime, Z}$ denotes the duality pairing between $Z^\prime$ and $Z$.
\end{lemma}

\begin{proof} It is straightforward to prove that $T^\prime$ is linear and bijective. To prove that it is an isometry, observe that
\begin{multline*}\|T^\prime \varphi\|_{Z_1^\prime} = \sup \left\{|\langle T^\prime \varphi,z_1\rangle_{Z_1^\prime,Z_1}|:\|z_1\|_{Z_1} = 1\right\} = \sup \left\{|\langle \varphi\circ T,z_1\rangle_{Z_1^\prime,Z_1}|:\|z_1\|_{Z_1} = 1\right\}\\ = \sup \left\{|\langle \varphi,T(z_1)\rangle_{Z_2^\prime,Z_2}|:\|z_1\|_{Z_1} = 1\right\} = \sup \left\{|\langle \varphi,z_2\rangle_{Z_2^\prime,Z_2}|:\|z_2\|_{Z_2} = 1\right\} = \|\varphi\|_{Z_2^\prime}.\end{multline*}
\end{proof}

\begin{theorem} If $X$ is reflexive, $\alpha\in(0,1]$ and $p,q\in(1,\infty)$ satisfy $(1/p) + (1/q) = 1$, then for each $\varphi\in (L_{\alpha}^{q}(0,T;X^\prime))^\prime$, there exists a unique $v_\varphi\in L_{\alpha}^{p}(0,T;X)$, such that
    \begin{equation}\label{202503131020}\langle\varphi,v\rangle_{(L_{\alpha}^{q}(0,T;X^\prime))^\prime,L_{\alpha}^{q}(0,T;X^\prime)}=\int_0^T\dfrac{(T-s)^{\alpha-1}}{\Gamma(\alpha)}\langle v(s),v_\varphi(s)\rangle_{X^\prime,X}\,ds,\end{equation}
    for all $v\in L_{\alpha}^{q}(0,T;X^\prime)$ and $\|\varphi\|_{(L_{\alpha}^{q}(0,T;X^\prime))^\prime}=\|v_\varphi\|_{L_{\alpha}^{p}(0,T;X)}$. 
\end{theorem}
\begin{proof} Let $\alpha\in(0,1)$. Recall the classical isometric isomorphism 
	\begin{equation}\label{202503121501}
			\begin{array}{c c l}
				T_{p,q}:(L^{p}(0,T;X))&\to& (L^{q}(0,T;X^\prime))^\prime\\
				f&\to&K_f,
			\end{array}\nonumber
		\end{equation}
where $K_f:L^{q}(0,T;X^\prime)\rightarrow\mathbb{R}$ is given by
$$K_f(g)=\int_0^T\langle g(s),f(s)\rangle_{X^\prime,X}\,ds,$$
for all $g\in L^{q}(0,T;X^\prime)$.

Next, recall that the proof of Theorem \ref{202503130949} ensures that for any $r>1$, the mapping $G_r:L_{\alpha}^{r}(0,T;X)\rightarrow L^{r}(0,T;X)$, given in \eqref{202503121439}, is an isometric isomorphism. Moreover, due to Lemma \ref{202503121452}, we already know that $G_r$ induces an isometric isomorphism between their dual spaces, which is given by
		\begin{equation}\label{202503121443}
			\begin{array}{c c l}
				G_r^\prime:(L^{r}(0,T;X^\prime))^\prime&\to& (L_{\alpha}^{r}(0,T;X^\prime))^\prime\\
				\psi&\to&\psi\circ H_r.
			\end{array}\nonumber
		\end{equation}

Finally consider the isometric isomorphism $S_{p,q}=G_q^\prime\circ T_{p,q}\circ G_p$ and observe that
		\begin{equation}\label{202503121443}
			\begin{array}{c c l}
				S_{p,q}:L_\alpha^{p}(0,T;X)&\to& (L_{\alpha}^{q}(0,T;X^\prime))^\prime\\
				f&\to& W_f.
			\end{array}\nonumber
		\end{equation}
where $W_f:L_\alpha^{q}(0,T;X^\prime)\rightarrow\mathbb{R}$ is given by
$$W_f(h)=\int_0^T\dfrac{(T-s)^{\alpha-1}}{\Gamma(\alpha)}\langle h(s),f(s)\rangle_{X^\prime,X}\,ds,$$
for all $h\in L_\alpha^{q}(0,T;X^\prime)$. In other words, the isometric isomorphism $S_{p,q}$ allows us to identify any $\varphi\in (L_{\alpha}^{q}(0,T;X^\prime))^\prime$ with a unique element $v_\varphi\in L_{\alpha}^{p}(0,T;X)$ such that $\|\varphi\|_{(L_{\alpha}^{q}(0,T;X^\prime))^\prime}=\|v_\varphi\|_{L_{\alpha}^{p}(0,T;X)}$ and \eqref{202503131020} holds.
\end{proof}

\begin{corollary}\label{202510271432} If $Y$ is a Hilbert space, $\alpha\in(0,1]$ and $p,q\in(1,\infty)$ satisfy $(1/p) + (1/q) = 1$, then for each $\varphi\in (L_{\alpha}^{q}(0,T;Y))^\prime$, there exists a unique $v_\varphi\in L_{\alpha}^{p}(0,T;Y)$, such that
    \begin{equation*}\langle\varphi,v\rangle_{(L_{\alpha}^{q}(0,T;Y))^\prime,L_{\alpha}^{q}(0,T;Y)}=\int_0^T\dfrac{(T-s)^{\alpha-1}}{\Gamma(\alpha)}(v(s),v_\varphi(s))_Y\,ds,\end{equation*}
    for all $v\in L_{\alpha}^{q}(0,T;Y)$ and $\|\varphi\|_{(L_{\alpha}^{q}(0,T;Y))^\prime}=\|v_\varphi\|_{L_{\alpha}^{p}(0,T;Y)}$. 
\end{corollary}

\section{The Fractional version of the unsteady Stokes Equations}
\label{galerfrac}

Based on the developments from the previous subsections, we now address the main questions introduced at the beginning of this article. This involves formulating a notion of weak solution to problem \eqref{navierstokes01} and proving its existence and uniqueness. We emphasize that the notations $H$ and $V$ (see \eqref{202503141707} and \eqref{202503141708}), used throughout this section, follow the classical framework employed in the study of the standard unsteady Stokes equations.
\begin{definition}[Weak Formulation] \label{202510151753}
For $f \in L^{2}(0, T; V')$ and $u_0 \in H$, we call a function 
$u \in L^{2}_{1-\alpha}(0, T; H) \cap L^{2}(0, T; V)$, satisfying 
$J_t^{1-\alpha} \|u(\cdot) - u_0\|_{H}^2 \in L^\infty(0,T)$, 
a \emph{weak solution} of~\eqref{navierstokes01} if 
\begin{equation}\label{NSF3}
  D_t^\alpha (u(t) - u_0, \eta)_H + \nu (u(t), \eta)_V 
  = \langle f(t), \eta \rangle_{V',V},
\end{equation}
for almost every $t \in [0,T]$ and for all $\eta \in V$, with $u(0) = u_0$.
\end{definition}

In the definition above, the function $u$ is required to belong to $L^{2}_{1-\alpha}(0,T;H) \cap L^{2}(0,T;V)$ and to satisfy $J_t^{1-\alpha}\|u(\cdot)-u_0\|_{V'}^2 \in L^\infty(0,T)$. Since functions of these spaces are defined only almost everywhere, the value $u(0) = u_0$ is not, {a priori}, well-defined. However, under the regularity assumptions above, one can prove that $u \in C([0,T];V')$, which ensures that the initial condition is well defined whenever $u_0 \in H \subset V'$, as a consequence of the continuous embeddings $V \hookrightarrow H \equiv H' \hookrightarrow V'$.

\begin{remark}
Although the regularity requirements introduced in Definition~\ref{202510151753} may seem stronger than those of the classical (non-fractional) framework, they play an analogous role. In fact, when $\alpha = 1$ we recover the standard setting, since $L^{2}_{1-\alpha}(0,T;H) = L^{\infty}(0,T;H)$ and $\left\|J_t^{1-\alpha}\|u(\cdot)-u_0\|_{H}^2\right\|_{L^{\infty}(0,T)} = \|u - u_0\|_{L^{\infty}(0,T;H)}^2$.
\end{remark}

We now show that the regularity assumed in the definition is sufficient to guarantee the continuity of $u$ in $V'$. To this end, we first present two auxiliary results.

\begin{lemma}\label{202503131728}
Consider $X$ a Banach space, $\alpha \in (0,1)$ and $p\in(1,\infty)$. If $v\in L^p(0,T;X)$ then
$$\|J_t^{1-\alpha} v(t)\|_X\leq \left( \frac{t^{1-\alpha}}{(1-\alpha)\Gamma(1-\alpha)} \right)^{(p-1)/p}
\big[ J_t^{1-\alpha} \|v(t)\|_X^p \big]^{1/p},$$
for almost every $t\in[0,T]$.
\end{lemma}

\begin{proof}
Notice that, 
\begin{multline*}
\|J_t^{1-\alpha} v(t)\|_X
\leq \frac{1}{\Gamma(1-\alpha)} \int_0^t (t-s)^{-\alpha(p-1)/p} (t-s)^{-\alpha/p} \|v(s)\|_X \, ds \\
\leq \frac{1}{\Gamma(1-\alpha)} \left( \int_0^t (t-s)^{-\alpha} \, ds \right)^{(p-1)/p}
\left( \int_0^t (t-s)^{-\alpha} \|v(s)\|_X^p \, ds \right)^{1/p} \\
\leq \left( \frac{t^{1-\alpha}}{(1-\alpha)\Gamma(1-\alpha)} \right)^{(p-1)/p}
\big[ J_t^{1-\alpha} \|v(t)\|_X^p \big]^{1/p},
\end{multline*}
for almost every $t\in[0,T]$.
\end{proof}

\begin{theorem}\label{202503131756} Consider a Banach space $X$, $\alpha\in(0,1)$, and $p\in(1,\infty)$ such that $\alpha>1/p$. Given $v_0\in X$ and a function $v\in L^p(0, T; X)$ satisfying $J_t^{1-\alpha}(v(\cdot)-v_0) \in W^{1,p}(0, T; X)$ and $J_t^{1-\alpha} \|v(\cdot)-v_0\|_X^p\in L^\infty(0,T)$, it follows that $v\in C([0,T],X)$ and there exists $g\in L^{p}(0,T;X)$ such that
\begin{equation}\label{202503131747}
v(t)=v_0+\dfrac{1}{\Gamma(\alpha)}\int_0^t(t-s)^{\alpha-1} g(s)\,ds,
\end{equation}
for every $t\in [0,T]$. Moreover, it holds that $v(0)=v_0$.
\end{theorem}

\begin{proof} 
Since $J_t^{1-\alpha}(v(\cdot)-v_0) \in W^{1,p}(0,T;X)$, item (iii) of Proposition \ref{Pro1} ensures that
\begin{equation}\label{202510091247}
J_t^\alpha\big[D_t^\alpha(v(t)-v_0)\big]
= v(t) - v_0 - \frac{t^{\alpha-1}}{\Gamma(\alpha)}
\Big\{ J_s^{1-\alpha}(v(s)-v_0) \Big\} \Big|_{s=0},
\end{equation}
for almost every $t \in [0,T]$. 
Since $J_t^{1-\alpha}(v(\cdot)-v_0) \in C([0,T];X)$ and $J_t^{1-\alpha} \|v(\cdot)-v_0\|_X^p\in L^\infty(0,T)$, by applying Lemma \ref{202503131728}, identity \eqref{202510091247} simplifies to
\begin{equation}\label{202503131750}
J_t^\alpha\big[D_t^\alpha(v(t)-v_0)\big] = v(t) - v_0,
\end{equation}
for almost every $t \in [0,T]$. 
Since $\alpha > 1/p$, \cite[Theorem 7]{CarFe1} guarantees that 
$J_t^\alpha\big[D_t^\alpha(v(\cdot)-v_0)\big] \in C([0,T];X)$, and also that
\begin{equation*}
J_t^\alpha\big[D_t^\alpha(v(t)-v_0)\big]\Big|_{t=0} = 0.
\end{equation*}
Thus, from \eqref{202503131750}, we conclude that $v \in C([0,T];X)$ and $v(0) = v_0$. 

To complete the proof, it remains to establish \eqref{202503131747}. For this purpose, define $g(t) = cD_t^\alpha v(t) = D_t^\alpha(v(t) - v_0)$, and rewrite \eqref{202503131750} accordingly.
\end{proof}

To proceed, we mimic the classical argument used when $\alpha = 1$. 
Recall that Fubini–Tonelli's theorem ensures that
$$
J_t^{1-\alpha}\big((u(t)-u_0),v\big)_H 
= \big(J_t^{1-\alpha}(u(t)-u_0),v\big)_H,
$$
which yields
\begin{equation*}
D_t^\alpha\big((u(t)-u_0),v\big)_H
= \frac{d}{dt}\Big[J_t^{1-\alpha}\big((u(t)-u_0),v\big)_H\Big]
= \frac{d}{dt}\Big[\big(J_t^{1-\alpha}(u(t)-u_0),v\big)_H\Big].
\end{equation*}
Following the approach used by Lions and Magenes in \cite{LiMa1} or Temam in \cite{RTemamb}, 
we then deduce from~\eqref{NSF3} that
\begin{equation}\label{202510271553}
\langle D_t^\alpha(u(t) - u_0), \eta \rangle_{V',V} 
= \langle f(t) - \nu A u(t), \eta \rangle_{V',V},
\end{equation}
where $A : V \to V'$ denotes the isomorphism defined by
\begin{equation*}
\langle A v_1, v_2 \rangle_{V',V} = (v_1, v_2)_V.
\end{equation*}

This implies that $D_t^\alpha (u(t) - u_0)$ belongs to $L^2(0,T;V')$, or equivalently, that $J_t^{1-\alpha}(u(\cdot) - u_0) \in W^{1,2}(0,T;V')$. Furthermore, since $u \in L^2_{1-\alpha}(0,T;H)$, it follows that $u \in L^2(0,T;V')$. Finally, since $J_t^{1-\alpha}\|u(\cdot)-u_0\|_{V'}^2 \in L^\infty(0,T)$, if $\alpha > 1/2$, we may apply Theorem \ref{202503131756} to conclude that $u \in C([0,T];V')$ and $u(0) = u_0$, given that $u_0 \in H$ and $H \subset V'$.

However, for $\alpha \in (0,1/2]$, the above arguments are not sufficient to ensure the continuity of the weak solution. 
This phenomenon is well known in the theory of fractional calculus, 
having been first identified by Hardy and Littlewood in their pioneering work \cite{HaLi1}, 
and later revisited in various contexts; see, for instance, \cite{AcGEcAn, CarFe2, CarFe1}.

\subsection{ The Approximate Solution} In this subsection, we present an approximate formulation for \eqref{NSF3}. To achieve this, let $\{w_j\}_{j=1}^\infty$ be a sequence of functions that is orthonormal in $H$, orthogonal in $V$, and complete in $V$; see \cite[Theorem IV.5.5]{FbPfMa}. For each $m \in \mathbb{N}$, we define the space $V_m = \text{span} \{w_1, \ldots, w_m\}$.

The fractional weak formulation of the approximation problem is defined as follows:  
given $m \in \mathbb{N}$, $\alpha \in (0,1)$, $f \in L^2(0,T;V')$, and $u_0 \in H$, 
we seek functions $\{g_{jm}\}_{j=1}^m$, with $g_{jm} : [0,T] \to \mathbb{R}$, such that by defining
\begin{equation}\label{Um}
  u_m(t) := \sum_{j=1}^m g_{jm}(t) w_i,
\end{equation}
we have
\begin{equation}\label{NSF3FA}
  \left\{
  \begin{aligned}
    cD_t^\alpha (u_m(t), w_i)_H + \nu (u_m(t), w_i)_V 
      &= \langle f(t), w_i \rangle_{V',V}, 
      && \forall i \in \{1, \ldots, m\}, \\[4pt]
    u_m(0) &= u_{0m},
  \end{aligned}
  \right.
\end{equation}
for all $t \in [0,T]$, where $u_{0m}$ denotes the orthogonal projection of $u_0$ onto $V_m$ in $H$, given by
\begin{equation*}
  u_{0m} = \sum_{j=1}^m (u_0, w_j)_H w_j 
  \longrightarrow u_0 
  \quad \text{in } H \text{ as } m \to \infty.
\end{equation*}

Note that using the definition of $u_{m}$ in the approximate problem given in (\ref{NSF3FA}), we obtain the following system of equations
\begin{gather*}
	\sum_{j=1}^{m}(w_{j},w_{i})_H\;cD_{t}^{\alpha}[g_{jm}(t)]+\nu\sum\limits_{j=1}^{m}(w_{j},w_{i})_Vg_{jm}(t)
	=\langle f(t),w_{i}\rangle_{V^\prime,V},\\
	g_{jm}(0)=(u_{0},w_{j})\;\; for\; j\in\{1,2,\ldots,m\}.
\end{gather*}
Since $\{w_{j}\}_{j=1}^{m}$ are orthonormals in $H$ and orthogonal in $V$, we have a system of nonlinear ordinary equations, given by
\begin{equation*}\left\{\begin{aligned}
	cD_{t}^{\alpha}U_{m}(t)+\nu B_{m}U_{m}(t)&=F_{m}(t),\\
	U_{m}(0)&=U_{0m},
\end{aligned}\right.\end{equation*}
with
\begin{align*}
	B_{m}=\left[\begin{array}{ccc}
		(w_{1},w_{1})_V&\ldots  & 0 \\
		\vdots       &\ddots  & \vdots\\
		0 &\ldots  & (w_{m},w_{m})_V
	\end{array}\right],\hspace*{1cm}
F_{m}(t)=\left[\begin{array}{c}
		\langle f(t),w_{1} \rangle_{V^\prime,V}\\
		\vdots  \\
		\langle f(t),w_{m} \rangle_{V^\prime,V}
	\end{array}\right],
\end{align*}\quad

\begin{align*}
U_{m}(t)=\left[\begin{array}{c}
		g_{1m}(t)\\
		\vdots  \\
		g_{mm}(t)
	\end{array}\right],\hspace*{1cm}
	U_{0m}=\left[\begin{array}{c}
		(u_{0},w_{1})_H\\
		\vdots  \\
		(u_{0},w_{m})_H
	\end{array}\right].
\end{align*}

Thus, we can  reinterpret the above system by the following Cauchy problem:
\begin{equation}\left\{\begin{array}{lll}
		\label{EDONavier-stokes3}
		cD_{t}^{\alpha}U_{m}(t)&=&G_{m}(t,U_{m}(t)),\;\; \textrm{ a.e. in }[0,T]\\
		U_{m}(0)&=&
		U_{0m}\in\mathbb{R}^{m},
	\end{array}\right.\end{equation}
with $G_{m}:[0,T]\times\mathbb{R}^m\rightarrow\mathbb{R}^m$ given by 
\begin{equation}\label{202503141756}G_{m}(t,x)=F_{m}(t)-\nu B_{m}x.\end{equation}

From the theory developed in Subsection \ref{caratsection}, we can directly establish the following theorem.

\begin{theorem}\label{THFWFA}
	Let $\alpha\in\left(1/2,1\right)$. Then \eqref{EDONavier-stokes3} has a unique global solution $U_m(t)$ in $[0,T]$. 
\end{theorem}
\begin{proof} Observe that $G_m:[0,T]\times\mathbb{R}^m\rightarrow\mathbb{R}^m$, defined in \eqref{202503141756}, satisfies the following conditions:
\begin{itemize}
\item[(i)] For each $x\in\mathbb{R}^m$, the function $G_m(\cdot,x):[0,T]\rightarrow\mathbb{R}^m$ is measurable.\vspace*{0.2cm}
\item[(ii)] For almost every $t\in[0,T]$, the function $G_m(t,\cdot):\mathbb{R}^m\rightarrow\mathbb{R}^m$ is continuous.
\end{itemize}
Thus, $G_m(t,x)$ is a Carath\'{e}odory map.

Now, observe that
\begin{equation*}
    \|G_m(t,x)-G_m(t,y)\|_{m}\leq \nu M\|x-y\|_{m},
\end{equation*}
for all $x,y\in \mathbb{R}^m$ and almost every $t\in[0,T]$. Moreover, we obtain
\begin{equation*}
    \|G_m(t,x)\|_m\leq \Big[\|w_1\|^2_V+\ldots+\|w_m\|^2_V\Big]^{1/2}\|f(t)\|_{V^\prime} + \nu\Big[\|w_1\|^2_V+\ldots+\|w_m\|^2_V\Big]^{1/2}\|x\|_m.
\end{equation*}

Applying Theorem \ref{stokes202501290918} for $p=q=2$, we ensure the existence and uniqueness of a continuous function $U_m:[0,T]\to\mathbb{R}^{m}$ that satisfies both equations in \eqref{EDONavier-stokes3}.
\end{proof}

It follows directly from Theorem \ref{THFWFA} that, for each $m \in \mathbb{N}$, there exist functions $\{g_{jm}\}_{j=1}^m \subset C([0,T];\mathbb{R})$ such that 
$\{cD_t^\alpha g_{jm}\}_{j=1}^m \subset L^2(0,T;\mathbb{R})$ (recall Proposition \ref{carat}). This, in turn, implies that $u_m \in C([0,T];V)$ and $cD_t^\alpha u_m \in L^2(0,T;V)$; see identity \eqref{Um}.

\subsection{Priori Estimates}\label{priori}It is important to emphasize that, from this point until the end of Section~\ref{galerfrac}, we shall assume $\alpha \in (1/2,1)$ to ensure the existence of the approximate solutions.

In analogy with the classical theory, multiplying the first equation in~\eqref{NSF3FA} by $g_{jm}$ for each $1 \le j \le m$ and summing over $j$ yields
\begin{equation}\label{Energy0}
  (cD_t^{\alpha} u_m(t), u_m(t))_H + \nu (u_m(t), u_m(t))_V 
  = \langle f(t), u_m(t) \rangle_{V',V},
\end{equation}
for almost every $t \in [0,T]$. 
In the classical case, i.e., when $\alpha = 1$, one typically proceeds by deriving the corresponding energy estimates using the identity
\begin{equation*}
  \dfrac{d}{dt} \|u_m(t)\|_H^2 = 2 (u_m^\prime(t), u_m(t))_H, 
  \quad \text{for a.e. } t \in [0,T].
\end{equation*}
However, adapting this approach becomes substantially more delicate in the presence of fractional derivatives. 
In a related direction, Carvalho-Neto et al. in \cite[Corollary 19]{CarFrOyTo1} established an auxiliary result that addresses this difficulty in the context of Caputo fractional derivatives. For completeness, we reproduce their result below.

\begin{proposition} \label{TFDF}
Let $0 < \beta < 1$ and suppose $v \in L^\infty(0,T; L^2(\Omega)) \cap C([0,\delta]; L^2(\Omega))$ for some $0 < \delta \leq T$, with $cD_t^\beta v \in L^2(0,T; L^2(\Omega))$. Then $cD_t^\beta \|v \|_{L^2(\Omega)}^2 \in L^1(0,T; \mathbb{R})$, and
\begin{equation*}
cD_t^\beta \left\| v(t) \right\|_{L^2(\Omega)}^2 \leq 2 \big( cD_t^\beta v(t), v(t) \big)_{L^2(\Omega)},
\end{equation*}
for almost every $t \in [0,T]$.
\end{proposition}

Thus, Proposition \ref{TFDF} guarantees that \eqref{Energy0} implies the inequality
\begin{equation*}
  \frac{1}{2} cD_t^\alpha \|u_m(t)\|_H^2 + \nu \|u_m(t)\|_V^2 
  \leq \langle f(t), u_m(t) \rangle_{V',V}, 
  \quad \text{for a.e. } t \in [0,T].
\end{equation*}
Then, by applying Young's inequality, we obtain
\begin{equation}\label{energy01}
  cD_t^\alpha \|u_m(t)\|_H^2 + \nu \|u_m(t)\|_V^2 
  \leq \frac{\|f(t)\|_{V'}^2}{\nu}, 
  \quad \text{for a.e. } t \in [0,T].
\end{equation}

Since for almost every $t\in[0,T]$ we have that
\begin{multline*} \int_0^t cD_t^{\alpha} \|u_m(s)\|_H^2\,ds=\int_0^t\dfrac{d}{ds}\left[J_t^{1-\alpha} \Big(\|u_m(s)\|_H^2-\|u_m(0)\|_H^2\Big)\right]\,ds\\
=J_t^{1-\alpha} \Big(\|u_m(s)\|_H^2-\|u_m(0)\|_H^2\Big),\end{multline*}
by integrating both sides of \eqref{energy01} from $0$ to $t$, we obtain
\begin{equation}\label{202510151837}
  J_t^{1-\alpha} \|u_m(t)\|_H^2 
  + \nu \int_0^t \|u_m(s)\|_V^2 \, ds 
  \leq \frac{t^{1-\alpha}}{\Gamma(2-\alpha)} \|u_0\|_H^2 
  + \frac{1}{\nu} \int_0^t \|f(s)\|_{V'}^2 \, ds, 
  \quad \forall \,t \in [0,T],
\end{equation}
which naturally implies
\begin{equation}\label{energy02}
  J_t^{1-\alpha} \|u_m(t)\|_H^2 \Big|_{t=T} 
  + \nu \|u_m\|_{L^2(0,T;V)}^2
  \leq \frac{T^{1-\alpha}}{\Gamma(2-\alpha)} \|u_0\|_H^2 
  + \frac{1}{\nu} \|f\|_{L^2(0,T;V')}^2.
\end{equation}
It follows from~\eqref{energy02} that $\{u_m\}_{m=1}^\infty \subset L_{1-\alpha}^2(0,T;H)\cap L^2(0,T;V)$ form bounded sequences in these spaces.

Now from \eqref{202510151837}, we can also deduce that
\begin{multline*}
  J_t^{1-\alpha} \|u_m(t)-u_{0m}\|_H^2\leq 2\big[J_t^{1-\alpha}\|u_m(t)\|_H^2+ J_t^{1-\alpha}\|u_{0m}\|_H^2\big] \\
  \leq \frac{4T^{1-\alpha}}{\Gamma(2-\alpha)} \|u_0\|_H^2 
  + \frac{2}{\nu} \int_0^T \|f(s)\|_{V'}^2 \, ds, 
  \quad \forall \,t \in [0,T],
\end{multline*}
what implies that $ \{J_t^{1-\alpha} \|u_m(t)-u_{0m}\|_H^2\}_{j=1}^\infty$ forms a bounded sequence in $L^\infty(0,T)$.

\subsection{Convergence and Passage to the Limit} By the reflexivity of $L_{1-\alpha}^2(0,T;H)$ (see Theorem~\ref{202503130949}) and $L^2(0,T;V)$, together with the boundedness of the sequence $\{u_m\}_{m=1}^\infty$ established at the end of the previous section, there exist functions $u \in L_{1-\alpha}^{2}(0,T;H)$ and $v \in L^{2}(0,T;V)$ such that, up to a subsequence,
\begin{equation}\label{202510161209}
  u_m \rightharpoonup u
  \quad \text{in } L_{1-\alpha}^{2}(0,T;H),
  \qquad 
  u_m \rightharpoonup v
  \quad \text{in } L^{2}(0,T;V).
\end{equation}

Indeed, the limits $u$ and $v$, obtained through different weak convergences, coincide.  Since we have $L^{2}(0,T;H)' \subset L_{1-\alpha}^{2}(0,T;H)' \cap L^{2}(0,T;V)'$, 
it follows from~\eqref{202510161209} that $u_m \rightharpoonup u$ and $u_m \rightharpoonup v$ in $L^{2}(0,T;H)$, and hence $u = v$.

On the other hand, the boundedness of $\{J_t^{1-\alpha}\|u_m(\cdot)-u_{0m}\|^2_H\}_{m=1}^\infty$ in $L^\infty(0,T)$, obtained at the end of the previous section, ensures the existence of a function $w \in L^{\infty}(0,T)$ such that, up to a subsequence,
\begin{equation}\label{202510271220}
  J_t^{1-\alpha}\|u_m(\cdot)-u_{0m}\|^2_H 
  \stackrel{*}{\rightharpoonup} w
  \quad \text{in } L^{\infty}(0,T).
\end{equation}

We now aim to prove that $J_t^{1-\alpha}\|u(t)-u_{0}\|^2_H = w(t)$ for almost every $t \in [0,T]$. 
To this purpose, it is necessary to establish a different type of convergence, namely that $u_m$ converges strongly in $L^{2}(0,T;H)$. 
For this, we recall the following result concerning compact subsets of $L^{p}(0,T;H)$, as proved by J.~Simon in~\cite[Theorem~5]{JSCompS}.
\begin{proposition}\label{CompacForH}
Suppose that $1 \le p < \infty$ and that:
\begin{itemize}
  \item[(i)] $W$ is bounded in $L^{p}(0,T;V)$;
  \vspace{0.1cm}
  \item[(ii)] $\|\tau_{h}u - u\|_{L^{p}(0,T-h;V')} \to 0$ uniformly for $u \in W$ as $h \to 0$.
  \vspace{0.1cm}
\end{itemize}
Then $W$ is relatively compact in $L^{p}(0,T;H)$. 
Here, $\tau_{h}$ denotes the shift operator defined by $(\tau_{h}u)(t) := u(t+h)$.
\end{proposition}

On the other hand, Li and Liu, in~\cite[Proposition~3.4]{LiLiu1}, established a criterion for the uniform convergence $\|\tau_{h}u - u\|_{L^{p}(0,T-h;V')} \to 0$, based on certain estimates for the Caputo fractional derivative of $u$. In their results, particular attention is given to the fact that the Caputo derivative must be properly defined at $t=0$ in order to make sense. To emphasize this aspect, we adapt their proof below, providing a version that fully addresses this issue.

\begin{proposition}\label{UniformComverV'}
There exists a constant $C>0$, independent of $h>0$ and $m \in \mathbb{N}$, such that for any $1\leq r < 1/(1-\alpha)$,
\begin{equation*}
  \|\tau_h u_m - u_m\|_{L^{r}(0,T-h;V')} \le C\|cD_{t}^{\alpha}u_{m}\|_{L^{1}(0,T;V')}^{r} h^{\alpha + \frac{1}{r} - 1}.
\end{equation*}
\end{proposition}
\begin{proof}
		Since $\{w_j\}_{j=1}^\infty$ is a complete system in $V$, the first equation in \eqref{NSF3FA} yields
\begin{equation}\label{202510271407}
  (cD_t^\alpha u_m(t), \eta)_H+\nu (u_m(t), \eta)_V
  = \langle f(t), \eta \rangle_{V',V},
  \quad \forall\, \eta \in V.
\end{equation}
Consequently,
\begin{equation*}
  \|cD_t^\alpha u_m\|_{L^{1}(0,T;V')}
  \le C \left(
    \|f\|_{L^{1}(0,T;V')}
    + \|u_m\|_{L^{1}(0,T;V)}
  \right),
  \end{equation*}
for some constant $C>0$ independent of $m$. Therefore, by inequality \eqref{energy02}, there exists $K>0$ such that
\begin{equation*}
  \|cD_t^\alpha u_m\|_{L^{1}(0,T;V')}
  \le K, \qquad \forall\, m \in \mathbb{N}.
\end{equation*}
Now, the regularity of $\{u_m\}_{m=1}^\infty$ and item~(v) of Proposition~\ref{Pro1} ensure that  
\begin{equation*}
  u_m(t)
  = u_{0m} + \frac{1}{\Gamma(\alpha)} 
    \int_0^t (t-s)^{\alpha-1} cD_s^\alpha u_m(s)\, ds,
  \quad \text{for a.e. } t \in [0,T].
\end{equation*}
Hence, for sufficiently small values of $h>0$, we compute
\begin{equation*}
  \tau_h u_m(t) - u_m(t)
 = \frac{1}{\Gamma(\alpha)} \Bigg(
    \int_t^{t+h} K_1(t-s,h)\, cD_s^\alpha u_m(s)\, ds
    - \int_0^t K_2(t-s,h)\, cD_s^\alpha u_m(s)\, ds
  \Bigg),
\end{equation*}
where $K_1(\tau,h) := (\tau + h)^{\alpha-1}$ and 
$K_2(\tau,h) := \tau^{\alpha-1} - (\tau + h)^{\alpha-1}$. 
Thus, we have
\begin{multline*}
  \int_0^{T-h} \|\tau_h u_m(t) - u_m(t)\|_{V'}^r\, dt
  \\\le \frac{2^{r-1}}{\Gamma(\alpha)^r}
    \Bigg[
      \int_0^{T-h}
        \Bigg(
          \int_t^{t+h} K_1(t-s,h) \|cD_s^\alpha u_m\|_{V'}\, ds
        \Bigg)^r dt \\
  + \int_0^{T-h}
        \Bigg(
          \int_0^t K_2(t-s,h) \|cD_s^\alpha u_m\|_{V'}\, ds
        \Bigg)^r dt
    \Bigg].
\end{multline*}

		Applying Hölder's inequality with the conjugate exponents $r$ and $r' = r/(r-1)$, we obtain
		\begin{multline*}
			\int_{t}^{t+h}K_{1}(t-s,h)\|cD_{s}^{\alpha}u_{m}(s)\|_{V'}\\\leq \left[\int_{t}^{t+h}\big(K_{1}(t-s,h)\big)^{r}\|cD_{s}^{\alpha}u_{m}(s)\|_{V'}ds\right]^{1/r}\left(\int_{t}^{t+h}\|cD_{s}^{\alpha}u_{m}(s)\|_{V'}ds\right)^{(r-1)/r}
		\end{multline*}	
		and 
		\begin{multline*}
			\int_{0}^{t}\big[K_{2}(t-s,h)\big]\|cD_{s}^{\alpha}u_{m}\|_{V'}\\\leq \left[\int_{0}^{t}[K_{2}(t-s,h)]^{r}\|cD_{s}^{\alpha}u_{m}(s)\|_{V'}ds\right]^{1/r}\left(\int_{0}^{t}\|cD_{s}^{\alpha}u_{m}(s)\|_{V'}ds\right)^{(r-1)/r},
		\end{multline*}
        and therefore that
\begin{multline}\label{202510271112}
  \int_0^{T-h} \|\tau_h u_m(t) - u_m(t)\|_{V'}^r\, dt
  \\\le \frac{2^{r-1}\|cD_{t}^{\alpha}u_{m}\|^{r-1}_{L^1(0,T;V')}}{\Gamma(\alpha)^r}
    \Bigg\{
      \int_{0}^{T-h}\left[\int_{t}^{t+h}\big(K_{1}(t-s,h)\big)^{r}\|cD_{s}^{\alpha}u_{m}(s)\|_{V'}ds\right]dt \\
  + \int_{0}^{T-h}\left[\int_{0}^{t}\big[K_{2}(t-s,h)\big]^{r}\|cD_{s}^{\alpha}u_{m}(s)\|_{V'}ds\right]dt
    \Bigg\}.
\end{multline}
        
		Furthermore, by the Fubini–Tonelli theorem, we have
		\begin{multline}\label{202510271113}
			\int_{0}^{T-h}\left[\int_{t}^{t+h}\big(K_{1}(t-s,h)\big)^{r}\|cD_{s}^{\alpha}u_{m}(s)\|_{V'}ds\right]dt\\
        =\int_{0}^{h}\|cD_{s}^{\alpha}u_{m}(s)\|_{V'}\left[\int_{0}^{s}\big(K_{1}(t-s,h)\big)^{r}dt\right]ds
       \\ +\int_{h}^{T-h}\|cD_{s}^{\alpha}u_{m}(s)\|_{V'}\left[\int_{s-h}^{s}\big(K_{1}(t-s,h)\big)^{r}dt\right]ds,
		\end{multline}
        and that
        \begin{multline}\label{202510271114}
			\int_{0}^{T-h}\left[\int_{0}^{t}\big[K_{2}(t-s,h)\big]^{r}\|cD_{s}^{\alpha}u_{m}(s)\|_{V'}ds\right]dt
        \\=\int_{0}^{T-h}\|cD_{s}^{\alpha}u_{m}(s)\|_{V'}\left[\int_{s}^{T-h}\big[K_{2}(t-s,h)\big]^{r}dt\right]ds.
		\end{multline}

        Now observe that:
        \begin{itemize}
        \item[(i)] for $s\in(0,h)$ we have
        \begin{equation}\label{202510271115}\int_{0}^{s}\big(K_{1}(t-s,h)\big)^{r}dt=\dfrac{h^{\alpha r}-(h-s)^{\alpha r}}{\alpha r}\leq\dfrac{h^{\alpha r}}{\alpha r};\end{equation}
        \item[(ii)] for $s\in(h,T-h)$ we have
        \begin{equation}\label{202510271116}\int_{s-h}^{s}\big(K_{1}(t-s,h)\big)^{r}dt=\dfrac{h^{\alpha r}}{\alpha r};\end{equation}
        \item[(iii)] for $s\in(0,T-h)$, since $(b-a)^q\leq b^q-a^q$ for $0\leq a\leq b$ and $1\leq q$ and $1\leq r<1/(1-\alpha)$, we may deduce the inequality
        \begin{equation}\label{202510271117}\int_{s}^{T-h}\big[K_{2}(t-s,h)\big]^{r}dt\leq \int_{s}^{T-h}\Big[(t-s)^{r(\alpha-1)}-(t+h-s)^{r(\alpha-1)}\Big]dt
			\leq\frac{h^{r(\alpha-1)+1}}{r(\alpha-1)+1}.
        \end{equation}
        \end{itemize}
        
        It follows from \eqref{202510271113}, \eqref{202510271114}, \eqref{202510271115}, \eqref{202510271116} and \eqref{202510271117}, when applied to inequality \eqref{202510271112}, that
        \begin{eqnarray}
			\int_{0}^{T-h}\|\tau_{h}u_{m}(t)-u_{m}(t)\|_{V'}^{r}\;dt&\leq&C\|cD_{s}^{\alpha}u_{m}\|_{L^{1}(0,T;V')}^{r}h^{r(\alpha-1)+1}.\nonumber
		\end{eqnarray}
\end{proof}

Proposition \ref{UniformComverV'} together with Proposition \ref{CompacForH} allow us to assert that for $\alpha\in(1/2,1)$ and $r=p=2$, the sequence $\{u_{m}\}_{m=1}^\infty$ is relatively compact in $L^{2}(0,T;H)$. That is, there exists a function $\tilde{u}\in L^{2}(0,T;H)$ such that, up to a subsequence,  
\begin{equation*}
	u_{m}\to \tilde{u}\hspace{0.3cm}\text{in}\hspace{0.3cm} L^{2}(0,T;H).
\end{equation*}

Thus, the aforementioned convergences~\eqref{202510161209} ensure that $\tilde{u} = u$. Moreover, thanks to the continuity of the operator $J_t^{1-\alpha}:L^2(0,T;H) \rightarrow L^2(0,T;H)$ (see details in \cite[Theorem 3.1]{CarFe0}), we deduce from~\eqref{202510271220} that $w(t) = \|u(t)-u_0\|_H^2$, for almost every $t\in[0,T]$. Summarizing the convergences obtained in this subsection, we have:
\begin{equation*}
	\left\{\begin{array}{l c c c l}
		u_{m}&\rightharpoonup&u&in&L_{1-\alpha}^{2}(0,T;H),\\
        u_{m}&\rightharpoonup& u&in&L^{2}(0,T;V),\\
        J_t^{1-\alpha}\|u_m(\cdot)-u_{0m}\|^2_H &\stackrel{*}{\rightharpoonup}& J_t^{1-\alpha}\|u(t)-u_0\|_H^2&in&L^{\infty}(0,T),\\
        u_{m}&\to& u&in&L^{2}(0,T;H).
        \end{array}\right.
\end{equation*}

The purpose of the remainder of this section is to prove that $u$ satisfies~\eqref{NSF3}. Let $\phi \in C_c^\infty([0,T];\mathbb{R})$. Multiplying~\eqref{202510271407} by $\phi$ and integrating over $(0,T)$, we obtain
\begin{equation*}
  \int_0^T (cD_t^\alpha u_m(t), \eta)_H\, \phi(t)\, dt 
  + \nu \int_0^T (u_m(t), \eta)_V\, \phi(t)\, dt
  = \int_0^T \langle f(t), \eta \rangle_{V',V}\, \phi(t)\, dt,
  \quad \forall\, \eta \in V.
\end{equation*}
This identity can be equivalently written as
\begin{equation}\label{202510271445}
  -\int_0^T (J_t^{1-\alpha} (u_m(t)-u_{0m}), \phi^\prime(t) \eta)_H\, dt 
  + \nu \int_0^T (u_m(t), \phi(t) \eta)_V\, dt
  = \int_0^T \langle f(t), \phi(t) \eta \rangle_{V',V}\, dt,
  \quad \forall\, \eta \in V.
\end{equation}

Since $\phi \eta\in L^{2}(0,T;V)$ and $u_m \rightharpoonup u$ in $L^{2}(0,T;V)$, it holds that
\begin{equation}\label{202510271455}
\int_{0}^{T}(u_{m}(t),\phi(t)\eta)_V\,dt
  \to
\int_{0}^{T}(u(t),\phi(t)\eta)_V\,dt.
\end{equation}

On the other hand, since $u_m \to u$ in $L^{2}(0,T;H)$, and recalling from \cite[Theorem~3.1]{CarFe0} the continuity of the operator $J_t^{1-\alpha}:L^2(0,T;H) \rightarrow L^2(0,T;H)$, together with $u_{0m}\to u_0$ in $H$, we deduce that
$$
J_t^{1-\alpha}(u_m - u_{0m}) \to J_t^{1-\alpha}(u - u_{0})
\quad \text{in } L^{2}(0,T;H).
$$
Consequently, it follows that
\begin{equation}\label{202510271507}
-\int_0^T (J_t^{1-\alpha} (u_m(t)-u_{0m}), \phi'(t) \eta)_H\, dt
\;\to\;
-\int_0^T (J_t^{1-\alpha} (u(t)-u_{0}), \phi'(t) \eta)_H\, dt.
\end{equation}

Thus, applying \eqref{202510271455} and \eqref{202510271507} to \eqref{202510271445}, we conclude that
\begin{equation*}
	-\int_{0}^{T}\big(J_{t}^{1-\alpha}(u(t)-u_{0}),\phi'(t)\eta\big)_H\,dt
  + \nu\int_{0}^{T}(u(t),\phi(t)\eta)_V\,dt
  = \int_{0}^{T}\langle f(t),\phi(t)\eta\rangle_{V',V}\, dt,
  \quad \forall\, \eta \in V.
\end{equation*}

Therefore, in the sense of distributions, we obtain that for all $v \in V$,
\begin{equation*}
	\frac{d}{dt}\big(J_{t}^{1-\alpha}(u(t)-u_{0}),v\big)_H
  + \nu (u(t),v)_V
  = \langle f(t),v\rangle_{V',V},
\end{equation*}
for almost every $t \in [0,T]$, which by Fubini–Tonelli’s theorem, is equivalent to
\begin{equation*}
	D_t^{1-\alpha}\big((u(t)-u_{0}),v\big)_H
  + \nu (u(t),v)_V
  = \langle f(t),v\rangle_{V',V},
\end{equation*}
for almost every $t \in [0,T]$, as desired.

\subsection{Uniqueness of Solution}

At the beginning of Section \ref{galerfrac} we proved that the regularity assumptions on the fractional weak solution, given in Definition \ref{202510151753}, are sufficient to guarantee that
$D_t^\alpha (u(t) - u_0) \in L^2(0,T;V')$ and (recall identity \eqref{202510271553})
\begin{equation}\label{202510271602}
D_t^\alpha (u(t) - u_0) + \nu A u(t) = f(t) \quad \text{in } V',
\end{equation}
for almost every $t \in [0,T]$. As a consequence, together with other results established therein, we obtained that $u \in C([0,T];V')$, which makes the initial condition $u(0) = u_0$ for $u_0 \in H$ meaningful, since $H \subset V'$.

We now prove the uniqueness of the weak solution. Let $u_1$ and $u_2$ be two weak solutions, and define $\tilde{u} = u_1 - u_2$. From \eqref{202510271602} it follows that  
\begin{equation}\label{202510271558}
D_t^\alpha \tilde{u}(t) + \nu A \tilde{u}(t) = 0,
\qquad
\tilde{u}(0) = 0,
\end{equation}
in $V'$. Applying $J_t^\alpha$ to both sides of \eqref{202510271558} and using item (iii) of Proposition \ref{Pro1}, we obtain  
\begin{equation*}
\tilde u(t) + \nu J_t^\alpha A \tilde u(t) = 0 \quad \text{in } V',
\end{equation*}
for almost every $t \in [0,T]$. Since $\tilde u \in L^2(0,T;V)$, taking the duality pairing with $\tilde u(t)$ yields  
\begin{equation*}
\langle \tilde u(t), \tilde u(t) \rangle_{V',V} 
+ \nu \big\langle J_t^\alpha A \tilde u(t), \tilde u(t) \big\rangle_{V',V} = 0,
\end{equation*}
for almost every $t \in [0,T]$. By definition of the duality pairings, this is equivalent to  
\begin{equation}\label{202510301553}
\|\tilde u(t)\|_H^2 
+ \nu \big( J_t^\alpha \nabla \tilde u(t), \nabla \tilde u(t) \big)_{\mathbb{L}^{2}(\Omega)} = 0,
\end{equation}
for almost every $t \in [0,T]$.

Since $\nabla \tilde u \in L^2(0,T;\mathbb{L}^{2}(\Omega))$, item (ii) of Proposition \ref{Pro1} gives  
$$D_t^\alpha [J^\alpha_t\nabla \tilde u(t)] = \nabla \tilde u(t),$$
for almost every $t \in [0,T]$. Hence, \eqref{202510301553} can be rewritten as  
\begin{equation}\label{202510301605}
\|\tilde u(t)\|_H^2 
+ \nu \Big( D_t^\alpha [J^\alpha_t\nabla \tilde u(t)], J_t^\alpha \nabla \tilde u(t) \Big)_{\mathbb{L}^{2}(\Omega)} = 0,
\end{equation}
for almost every $t \in [0,T]$.  

Because $\alpha > 1/2$, Theorem 7 in \cite{CarFe1} ensures that $J^\alpha_t\nabla \tilde u \in C([0,T];\mathbb{L}^{2}(\Omega))$, allowing us to apply Proposition \ref{TFDF} to \eqref{202510301605}. We then conclude that $D_t^\alpha \|J_t^\alpha \nabla \tilde u(t)\|_V^2 \in L^1(0,T)$ and  
\begin{equation}\label{202510301611}
\|\tilde u(t)\|_H^2 
+ (\nu/2) D_t^\alpha \|J_t^\alpha \nabla \tilde u(t)\|_V^2 \leq 0,
\end{equation}
for almost every $t \in [0,T]$. Finally, by item (iii) of Proposition \ref{Pro1}, \eqref{202510301611} is equivalent to  
\begin{equation*}
J_t^\alpha \|\tilde u(t)\|_H^2 
+ (\nu/2) \|J_t^\alpha \nabla \tilde u(t)\|_V^2 \leq 0,
\end{equation*}
for almost every $t \in [0,T]$. Hence $J_t^\alpha \|\tilde u(t)\|_H^2 = 0$ a.e.\ in $[0,T]$. Since $J_t^\alpha : L^2(0,T) \to L^2(0,T)$ is injective (if follows from item (ii) of Proposition \ref{Pro1}), we conclude that $\tilde{u} = 0$, or equivalently, $u_1 = u_2$. This proves uniqueness.

\section{Closing Remarks}

Throughout this study, several questions have remained open, and we would like to highlight some of them below. We emphasize that part of these topics are currently under investigation and will be addressed in future work.

\begin{itemize}
\item[(i)] In analogy with the classical case $\alpha = 1$, can one prove that the weak solution $u$ possesses improved continuity properties, namely $u \in C([0,T];H)$?
\item[(ii)] A natural strategy for establishing uniqueness would be to generalize Lemma~1.2 from Chapter 3 of Temam’s classical monograph \cite{RTemamb} in the spirit of Proposition \ref{TFDF}, following the same line of reasoning as in the standard case. A result of this nature was already proved by Carvalho-Neto and Fehlberg Júnior in \cite[Theorem 4.16]{PaRFL}; however, their approach requires the additional assumption $J_t^{1-\alpha}\|u(t)\|_H^2 \in W^{1,1}(0,T)$, a regularity property that we have not been able to obtain in our setting.
\item[(iii)] If stronger regularity is imposed on the initial condition, for instance $u_0 \in V$, what additional regularity properties for $u$ could then be derived?
\item[(iv)] How can one handle the nonlinear term in the fractional Navier–Stokes system, given that even the Stokes case already presents substantial analytical challenges that are not yet fully understand?
\end{itemize}

\end{document}